\documentclass[10pt]{amsart}%
\usepackage{amsfonts}
\usepackage[T1]{fontenc}
\usepackage[english]{babel}
\usepackage{graphicx}
\usepackage{amsmath}
\usepackage{amssymb}
\usepackage{amsthm}
\usepackage{nicefrac}
\usepackage{hyperref}
\usepackage{lineno }
\providecommand{\U}[1]{\protect\rule{.1in}{.1in}}
\usepackage{tikz}

\newtheorem{example}{Example}

\newtheorem{proposition}{Proposition}
\newtheorem{lemma}{Lemma}
\newcommand{\df}{{F}}
\newcommand{\ldf}{{\underline{F}}}
\newcommand{\udf}{{\overline{F}}}
\newcommand{\pbox}{{[\ldf,\udf]}}
\newcommand{\setpb}{{\Phi(\ldf,\udf)}}

\begin{document}
\title[expectations and p-boxes]{Computing expectations with continuous p-boxes: univariate case}
\author[L. Utkin]{Lev Utkin}
\address{St. Petersburgh Forest Technical Academy, Dept. of Computer Science, Institutski per. 5, 194021, St. Petersburgh, Russia}
\email{lev.utkin@list.ru}
\author[S. Destercke]{Sebastien Destercke}
\address{Institut de Radioprotection et
de Suret\'e\ Nucl\'{e}aire (IRSN), Cadarache,
France} \email{desterck@irit.fr}

\thanks{This paper is an extended version of the first part of~\cite{UtkinDest07}}

\begin{abstract}
Given an imprecise probabilistic model over a continuous space,
computing lower/upper expectations is often computationally hard
to achieve, even in simple cases. Because expectations are essential in decision making and risk analysis, tractable methods to compute them are crucial in many applications involving imprecise probabilistic models. We
concentrate on p-boxes (a simple and popular model), and on the computation of lower
expectations of non-monotone functions. This paper is devoted to the univariate case, that is where only one variable has uncertainty. We propose and compare two approaches : the first using general linear programming, and the second using the fact that p-boxes are
special cases of random sets. We underline the complementarity of
both approaches, as well as the differences.
\end{abstract}
\maketitle

\section{Introduction}

There are many situations where a unique probability distribution cannot be identified to describe our uncertainty about the value assumed by a variable on a state space. This can happen for example when data or expert judgments are not sufficient and/or are conflicting. In such cases, a solution is to model information by the means of imprecise probabilities, that is by considering either sets of probability distributions~\cite{levi1980a,Huber81} or  bounds on expectations~\cite{Miranda08}. Note that, from a purely mathematical point of view, such representations encompass many other frameworks dealing with the representation of incomplete and conflicting information, such as random sets~\cite{Dempster67} and possibility theory~\cite{DuboisPrade92bis}.

When considering such models, the expectation of a real-valued bounded function over the state space is no longer precise and is lower- and upper-bounded by some value. In applications involving risk analysis or decision making, the decision process will be based on the values of these lower and upper expectations, using extensions of the classical expected utility criterion~\cite{Troffaes07}.  When the state space on which the variable assumes its value is finite, lower and upper expectations can be numerically computed by using, for instance, linear programming techniques~\cite{UtkinAug05}. The problem becomes quite more complicated when uncertainty models are defined over infinite state spaces (e.g., the real line, product spaces, \ldots). 

In this latter case, computing exactly and analytically the lower and upper expectations of a given function is impossible most of the time, and there are very few methods and algorithms around to compute approximations of these bounds~\cite{Cozman00,ObermeierAugusti07,Troffaes08}. In this paper, we study such analytical solutions for a specific case, that is the one where the uncertainty over a variable is described by a pair of upper and lower cumulative distributions (a so-called p-box~\cite{FersonAll03}). In essence, such a study comes down to search the extremal points of the p-box for which the expectation bounds are reached. The features of these solutions also allow us to suggest some ways to build more efficient numerical methods and algorithms, useful when analytical solutions cannot be computed. We also assume that the function over which lower and upper expectations have to be computed can be non-monotone but has a (partially) known behaviour.  In this paper, we concentrate on the univariate case, i.e., where the value assumed by only one variable is tainted with uncertainty. The multivariate case as well as the case of mixed strategies (expectation bounds computed over mixture of functions) are left for forthcoming papers.

P-boxes are one of the simplest and most popular models of sets of
probability distributions, directly extending cumulative
distributions used in the precise case. P-boxes are often used in applications~\cite{KrieglerHeld05}, as they can be easily derived from small samples~\cite{BaudritDubois06} or from expert opinions expressed in terms of imprecise percentiles. consequently, our study is likely to be useful in many practical situations. P-box models can also be found in robust Bayesian analysis, where they are known as distribution band classes~\cite{BasuDasGupta95}. In other cases, the poor expressiveness of p-boxes compared to more general sets of probabilities is clearly a limitation~\cite{DesterckeAll07IJAR1}. However, as we shall see, their simplicity allows for more efficient computations, and they can provide quick first approximations. Eventually, if these first approximations already allow to take a decision, there is no need to consider more complex (and computationally demanding) models.

Methods developed in the paper are based on two different
approaches, and we found it interesting to emphasize similarities and
differences between these approaches, as well as how one approach
can help the other: the first is based on the fact that the computation of bounding expectations can be viewed as a linear programming problem, while the
second uses the fact that a p-box is a particular case of a random set~\cite{KrieglerHeld05,DesterckeAll07IJAR1}. Approximating lower and upper expectations with these approaches mainly consists in discretizing the uncertainty models. In this sense, they are different from other approaches discretizing the state space~\cite{ObermeierAugusti07,Troffaes08}.

We first state the general problem in Section~\ref{sec:probstat}, how to solve it by using linear programming and random sets, and introduce the problem of conditioning by an observed event. We then study the computation of lower/upper expectations of a function over the p-box for different behaviours. Going from the simplest case to the most general one, we start with monotone functions in Section~\ref{sec:mon-func}, pursue with functions having one extrema in Section~\ref{sec:maxunivar}, and finish by general (bounded) continuous functions in Section~\ref{sec:maxminunivar}.

\section{General problem statement \label{sec:probstat}}

We assume that the information about a (real-valued) random variable $X$ is (or
can be) represented by a lower
$\underline{F}$ and upper $\overline{F}$ cumulative probability distributions
defining the p-box $[\underline
{F},\overline{F}]$~\cite{FersonAll03}. Lower $\underline{F}$ and
upper $\overline{F}$ distributions thus define a set $\setpb$ of precise distributions such that
\begin{equation}
\setpb = \{ F | \forall
x\in\mathbb{R}, \ \underline{F}(x)\leq F(x)\leq\overline{F}(x) \}.
\label{NonMonF30}%
\end{equation}

Given a function $h(X)$, lower ($\underline{\mathbb{E}}$) and upper
($\overline{\mathbb{E}}$) expectations over
$[\underline{F},\overline{F}]$ of $h(X)$ can be computed by means of
a procedure sometimes called natural extension~\cite{Walley91,Walley96}, which
corresponds to the following equations:
\begin{align}
\underline{\mathbb{E}}(h)=\inf_{\df \in \setpb }\int_{\mathbb{R}}h(x)\mathrm{d}F  &  , \overline{\mathbb{E}%
}(h)=\sup_{\df \in \setpb }\int_{\mathbb{R}%
}h(x)\mathrm{d}F. \label{NonMonF32}%
\end{align}
Computing the lower (resp. upper) expectation can be seen as finding
the extremizing distribution $F$ inside $\setpb$ reaching the infimum (resp. supremum) in
Equations (\ref{NonMonF32}). If we consider the convex set of probabilities induced by $\setpb$, this is equivalent to find the extremum point (i.e., vertex) of this convex set where the bounds are reached, among all vertices (here infinitely many). Solving Equations~\eqref{NonMonF32} exactly is usually very difficult, although sometimes possible, even when analytical expressions of $h,\udf,\ldf$ are known. In practice, numerical methods must often be used to solve the problem and estimate both the upper and lower expectations. Upper and lower expectations are dual~\cite[ch.2.]{Walley96}, in the sense that $\underline{\mathbb{E}}(h)=-\overline{\mathbb{E}}(-h)$. This will allow us to concentrate only on the lower expectations for some cases studied in the sequel.  We now detail the two generic approaches used throughout the paper to solve the above problem. Note that, through all the paper, we assume that we restrict ourselves either to $\sigma$-additive probabilities or to continuous functions $h$, as such assumptions are not, from a practical standpoint, very limiting.

We will denote by $I_{A}$ the indicator function of the set $A$, that is the function such that $I_A(x)=1$ if $x \in A$, zero otherwise. The lower (resp. upper) expectation of this function, $\underline{\mathbb{E}}(I_A)$ (resp. $\overline{\mathbb{E}}(I_A)$), have the same value as the lower (resp. upper) probability $\underline{P}(A)$ (resp. $\overline{P}(A)$ of the event $A$ induced by the set $\setpb$.

\subsection{Linear programming view}
\label{sec:linprog-view}

Although we assume that the readers have basic knowledge of linear programming (for an introduction to the topic, see for example Vanderbei~\cite{Vanderbei07}), we will recall basic results coming from this theory when they are used in the paper. 

As sets of probabilities can be expressed through linear constraints over expectations, and as expectation is a linear functional, it is quite natural to translate Equations~\eqref{NonMonF32} into linear programs. The linear programs corresponding to lower expectation are summarized below.
\begin{equation*}
\begin{array}{c@{\hspace{0.5cm}}c@{\hspace{0.5cm}}c}
 \cline{1-1} \cline{3-3}
\textrm{\textbf{Primal problem:}} & &  \textrm{\textbf{Dual problem:}} \\
 \cline{1-1} \cline{3-3} \\ \textrm{Min.} \ \
\mathbf{v}=\int\limits_{-\infty}^{\infty}h\left(  x\right)  \rho\left(  x\right)
\mathrm{d}{x} & &
 \textrm{Max.} \ \
\mathbf{w}=c_{0}+\int\limits_{-\infty}^{\infty}\left(  -c\left(  t\right) \overline
{F}\left(  t\right)  +d\left(  t\right) \underline{F}\left(
t\right)
\right)  \mathrm{d}{t}\\  & & \\
\textrm{subject to} & & \textrm{subject to}\\ & & \\
\rho\left(  x\right)  \geq0, \int\limits\limits_{-\infty}^{\infty}\rho\left(
x\right)
\mathrm{d}{x}=1, & &
 c_{0}+\int\limits_{x}^{\infty}\left(  -c\left(  t\right)  +d\left(
t\right)
\right)  \mathrm{d}{t}\leq h\left(  x\right)  , \\
-\int\limits_{-\infty}^{x}\rho\left(  x\right) \mathrm{d}{x}\geq-\overline
{F}\left(  x\right)  ,& &
 c_{0}\in\mathbb{R}, c\left(  x\right)  \geq0, d\left(  x\right)  \geq0.\\
\int\limits_{-\infty}^{x}\rho\left(  x\right)
\mathrm{d}{x}\geq\underline{F}\left( x\right)  . & &\\ & &\\ \cline{1-1} \cline{3-3}
\end{array}
\end{equation*}
Where $\mathbf{v}$ and $\mathbf{w}$ are the objective functions to respectively minimize and maximize for the primal and dual problems, and $\rho\left(  x\right)$ is a probability density function having a cumulative distribution inside $\setpb$. Since both the primal and dual problems are feasible (i.e. have solutions satisfying their constraints), then their optimal solutions coincide (due to strong duality~\cite[Ch.5]{Vanderbei07}) and are equal to $\underline{\mathbb{E}}(h)$. 

Numerically solving the above problem can be done by approximating
the probability distribution function $F$ by a set of $N$ points
$F(x_{i})$, $i=1,...,N$, and by translating equations
(\ref{NonMonF32}) into the corresponding linear programming problem
with $N$ optimization variables and where constraints correspond to
equation (\ref{NonMonF30}). Those linear programming problems are of
the form
\begin{align} \label{eq:primal-approx}
\underline{\mathbb{E}}^{\ast}(h)
=\inf\sum_{k=1}^{N}\!h(x_{k})z_{k} \;
\text{ or } \; \overline{\mathbb{E}}^{\ast}(h) =\sup\sum_{k=1}%
^{N}\!h(x_{k})z_{k} %
\end{align}
subject to
\begin{align*}
&  z_{i}\geq0,\;\ i=1,...,N,\ \sum_{k=1}^{N}z_{k}=1,\\
&  \sum_{k=1}^{i}z_{k}\leq\overline{F}(x_{i}),\ \sum_{k=1}^{i}z_{k}%
\geq\underline{F}(x_{i}),\ i=1,...,N.
\end{align*}
where the $z_{k}$ are the optimization variables, and objective
function $\underline{\mathbb{E}}^{\ast}(h)$
(resp. $\overline{\mathbb{E}}^{\ast}(h)$) is an approximation of the lower
(resp. upper) expectation. Note that the primal problem may not always be feasible (e.g., consider $N=1$ and $\overline{F}(x_{1})-\underline{F}(x_{1})< 1$) if $N$ is too small or values $x_i$ are badly chosen.  Also, the inequality $\underline{\mathbb{E}}(h) \leq \underline{\mathbb{E}}^{\ast}(h)$ (or its converse) does not always hold when solving the above discretized problem. The approximated solution $\underline{\mathbb{E}}^{\ast}$ is thus not a guaranteed inner or outer approximation. A solution to obtain a guaranteed inner approximation is to replace, for $i=1,\ldots,N$,  $\underline{F}(x_{i})$ by $\underline{F}(x_{i+1})$ in constraints $\sum_{k=1}^{i}z_{k} \geq\underline{F}(x_{i})$, with $\underline{F}(x_{N+1})=1$, since in this case, any solution to the linear program would be such that, for any $x \in [x_i,x_{i+1}]$,
$$\underline{F}(x) \leq \underline{F}(x_{i+1}) \leq \sum_{k=1}^i z_k  \leq \overline{F}(x_i) \leq \overline{F}(x),$$
consequently the (discrete) cumulative distributions formed by the values $z_k$, $k=1,\ldots,N$ is in $\setpb$. However, for this linear program to have a solution, we must be able to choose the $x_i$, $i=1,\ldots,N$ on $\mathbb{R}$ such that $\overline{F}(x_i) \geq \underline{F}(x_{i+1})$. In addition to not be always possible, this puts necessary constraints over the chosen discretization of $\mathbb{R}$.

Let us write now the dual linear programming problem for computing
$\underline{\mathbb{E}}^{\ast\ast}(h)$, taking points $y_{i}$ different from
$x_{i}$,%
\begin{equation} 
\underline{\mathbb{E}}^{\ast\ast}(h)=\max\left(  c_{0}+\sum_{i=1}^{N}\left(
d_{i}\underline{F}\left(  y_{i}\right)  -c_{i}\overline{F}\left(
y_{i}\right)  \right)  \right)  \label{eq:NumerLowerprev2}%
\end{equation}
subject to $c_{0}\in\mathbb{R}$, $c_{i}\geq0$, $d_{i}\geq0$, and%
\[
c_{0}+\sum_{k=i}^{N}\left(  d_{k}-c_{k}\right)  \leq h(y_{i}),\ i=1,...,N,
\]
where $c_{0}$, $c_{i}$, $d_{i}$ are the optimization variables, $y_{i}%
=(x_{i-1}+x_{i})/2$. 

When both problems are discretized, equality between their optimal solutions no longer holds, but converge towards the same value as $N$ grows. To approximate the solution, one can let $N$ grow iteratively until the difference $\left\vert \underline{\mathbb{E}}^{\ast}(h)-\underline{\mathbb{E}}^{\ast\ast}(h)\right\vert$ is smaller than a given value $\varepsilon>0$ characterizing the accuracy of
the solutions. However, this way of determining the lower and upper
expectations meets some computation difficulties if many iterations are needed and if the value of
$N$ is rather large. Indeed, the primal optimization problem have $N$
variables and $3N+1$ constraints. On the other hand, solving the primal and dual approximated problems only once with a
small value of $N$ can lead to bad approximations
of the exact value. Also important is the question of how to choose or sample the values $x_i$ to improve numerical convergence? In other words, is there some regions that should be more sampled than others. A generic algorithm (for $\underline{\mathbb{E}}$) would look as follows:
\begin{enumerate}
\item Fix a precision threshold $\epsilon$ and an initial value of $N$
\item Sample $N$ values $x_i$ s.t. $\udf(x_i)>0$ and $\ldf(x_i)<1$
\item Compute $\underline{\mathbb{E}}^{\ast}(h)$ and $\underline{\mathbb{E}}^{\ast\ast}(h)$
\item If $\left\vert \underline{\mathbb{E}}^{\ast}(h)-\underline{\mathbb{E}}^{\ast\ast}(h)\right\vert \leq \epsilon$, stop, else increase $N$ and return to step 2.
\end{enumerate}

In the sequel, we will see that knowing $h$ and its behaviour can significantly improve both accuracy and efficiency of expectation bound computations. It also provides some insight as to how values $x_i$ could be sampled. 

\subsection{Random set view}

Now that we have given a global sketch of the linear programming
approach, we can detail the one using random sets. Formally, a
random set is a mapping $\Gamma$ from a probability space to the
power set $\wp(X)$ of another space $X$, also called a multi-valued
mapping. This mapping induces lower and upper probabilities on
$X$~\cite{Dempster67}. Here, we consider the unit interval
$[0,1]$ equipped with Lebesgue measure as the probability space, and $\wp(X)$ are
the measurable subsets of the real line $\mathbb{R}$.

Given the p-box
$[\underline{F},\overline{F}]$, we will denote
$A_{\gamma}=[a_{\ast\gamma},a_{\gamma}^{\ast}]$ the set such that
\begin{align*}
a_{\ast\gamma}:=\sup\{x\in \mathbb{R}:\overline{F}(x)<\gamma
\}=\overline{F}^{-1}(\gamma)  &  ,\\
a_{\gamma}^{\ast}:=\inf\{x\in \mathbb{R}:\underline{F}%
(x)>\gamma\}=\underline{F}^{-1}(\gamma)  &  ,
\end{align*}
\begin{center}
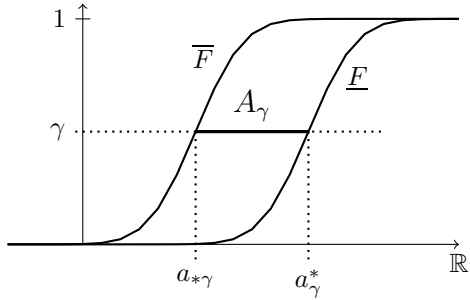
\begin{figure}[htb]
\centering
\begin{tikzpicture}

\draw[->] (-1,0) -- (5,0) node[below] {$\mathbb{R}$};
\draw[->] (0,0) -- (0,3.2);
\draw (-0.1,3) node[left] {$1$} -- (0.1,3);
\draw[domain=-1:5,thick]  plot[id=norm] function{3*norm((x-1.5)*2)} ;
\node at (1.6,2.5) {$\udf$};
%\draw[domain=-1:5,ultra thick,dashed]  plot[id=norm] function{3*norm((x-1)*2)};
\draw[domain=-1:5,thick]  plot[id=norm2] function{3*norm((x-3)*2)};

\draw[thick,dotted] (-0.1,1.5) node[left] {$\gamma$} -- (4,1.5) ; 
\draw[thick,dotted] (1.5,1.5) -- (1.5,-0.2) node[below] {$a_{\ast\gamma}$} ;
\draw[thick,dotted] (3,1.5) -- (3,-0.2) node[below] {$a^{\ast}_\gamma$};
\draw[very thick] (1.5,1.5) -- (3,1.5) node[midway,above] { \Large $A_{\gamma}$};
\node at (3.65,2.2) {$\ldf$};
\end{tikzpicture}
\caption{P-box as random set,
illustration}%
\label{fig:pboxrands}%
\end{figure}
\end{center}
By extending existing results~\cite{KrieglerHeld05,FersonAll03} to the
continuous real line~\cite{DesterckeAll07IJAR2,Alvarez06}, we can conclude that the p-box $[\underline
{F},\overline{F}]$ is equivalent to the continuous random set with a
uniform mass density on $[0,1]$ and a mapping (see
figure~\ref{fig:pboxrands}) such that
\[
\Gamma(\gamma)=A_{\gamma}=[a_{\ast\gamma},a_{\gamma}^{\ast}],\;\gamma\in
\lbrack0,1].
\]
Note that both $\overline{F}^{-1}(\gamma),\underline{F}^{-1}(\gamma)$ are non-decreasing functions of $\gamma$. The interest of this mapping $\Gamma$ is that it allows us to rewrite
equations (\ref{NonMonF32}) in the following form:
\begin{align}
&  \underline{\mathbb{E}}(h)=\int_{0}^{1}\inf_{x\in
A_{\gamma}}h(x)\;d\gamma
,\label{eq:Lowerprevleb}\\
&  \overline{\mathbb{E}}(h)=\int_{0}^{1}\sup_{x\in
A_{\gamma}}h(x)\;d\gamma.
\label{eq:upperprevleb}%
\end{align}

Again, finding analytical solutions of such integrals is not easy in the
general case, but numerical approximations can be computed (with more or less difficulty) by discretizing the p-box on a finite number
of levels $\gamma_{i}$, the main difficulty in the general case
being to find the infimum or supremum of $h(X)$ for each discretized
level.  Note that, in the finite case, a random set can be represented by non-null weights, here denoted $m$, given to subsets of space $X$ and summing up to one (i.e., $\sum_{E \subseteq X} m(E)=1$).  Let $\gamma_{0} = 0 \leq \gamma_{1}\leq \ldots \leq \gamma_{M} = 1$ and define the  discrete random set $\overline{\Gamma}$ such that for $i=1,\ldots,M$
$$\overline{\Gamma} := \left\{ \begin{array}{c} A_{\overline{\gamma_i}}=[a_{\ast\gamma_{i-1}},a_{\gamma_i}^{\ast}], 
 \\ m(A_{\overline{\gamma_i}})=\gamma_i - \gamma_{i-1} \end{array} \right.$$
We denote by $\setpb_{\overline{\Gamma}}$ the set of precise distributions induced by $\overline{\Gamma}$. This discretization, which is an outer approximation of the p-box $\pbox$ (i.e., $\setpb \subset \setpb_{\overline{\Gamma}}$), is sometimes referred to as the ODM (Outer discretization Method) and has been studied by other authors~\cite{Tonon08}. Working with $\overline{\Gamma}$, Equations~\eqref{eq:Lowerprevleb}, \eqref{eq:upperprevleb} can be rewritten as
$$ \underline{\mathbb{E}}^{\overline{\Gamma}}(h)= \sum_{i=1}^M m(A_{\overline{\gamma_i}}) \inf_{x \in A_{\overline{\gamma_i}}} h(x) \textrm{ and } \overline{\mathbb{E}}^{\overline{\Gamma}}(h)= \sum_{i=1}^M m(A_{\overline{\gamma_i}}) \sup_{x \in A_{\overline{\gamma_i}}} h(x).$$
Let us now define another discrete random set $\underline{\Gamma}$ such that for $i=1,\ldots,M$
$$\underline{\Gamma} := \left\{ \begin{array}{c} A_{\underline{\gamma_i}}= [a_{\ast\gamma_{i}},a_{\gamma_{i-1}}^{\ast}] \textrm{ if } a_{\ast\gamma_{i}}\leq a_{\gamma_{i-1}}^{\ast}, \ \ \emptyset \textrm{ otherwise } 
 \\ m(A_{\underline{\gamma_i}})=\gamma_i - \gamma_{i-1} \end{array} \right.$$
We denote by $\setpb_{\underline{\Gamma}}$ the set of precise distributions induced by $\underline{\Gamma}$. $\underline{\Gamma}$  is an inner approximation of the p-box (i.e., $\setpb_{\underline{\Gamma}} \subset \setpb$), and Equations\eqref{eq:Lowerprevleb}, \eqref{eq:upperprevleb} can again be rewritten
$$ \underline{\mathbb{E}}^{\underline{\Gamma}}(h)= \sum_{i=1}^M m(A_{\underline{\gamma_i}}) \inf_{x \in A_{\underline{\gamma_i}}} h(x) \textrm{ and } \overline{\mathbb{E}}^{\underline{\Gamma}}(h)= \sum_{i=1}^M m(A_{\underline{\gamma_i}}) \sup_{x \in A_{\underline{\gamma_i}}} h(x).$$
Note that when there is an index $i$ for which $A_{\underline{\gamma_i}}=\emptyset$, $\underline{\Gamma}$ does no longer describe a non-empty set of probabilities, and we will name such a random set inconsistent. This case can be compared to the case when the linear program giving guaranteed inner approximation has no feasible solutions. 

We have that  $\underline{\mathbb{E}}^{\overline{\Gamma}}(h) \leq \underline{\mathbb{E}}(h) \leq \underline{\mathbb{E}}^{\underline{\Gamma}}(h)$ (due to inclusions $\setpb_{\underline{\Gamma}} \subset \setpb \subset \setpb_{\overline{\Gamma}} $ ). Thus, to approximate the solution we can again let $M$ grow until $\vert \underline{\mathbb{E}}^{\overline{\Gamma}}(h) - \underline{\mathbb{E}}^{\underline{\Gamma}}(h) \vert$ is smaller than a given accuracy $\varepsilon>0$. As in the case of linear programming, choosing too few levels
$\gamma_{i}$ or using poor heuristics to find the infinimum/supremum over sets can lead to bad
approximations, and if those infinimum/supremum are hard to find, computational difficulties can arise. A generic algorithm (for $\underline{\mathbb{E}}$) using random sets would be as follows \begin{enumerate}
\item Fix a precision threshold $\epsilon$ and an initial value of $M$
\item Sample $M$ values $\gamma_i$
\item Compute $\underline{\mathbb{E}}^{\overline{\Gamma}}(h)$ and $\underline{\mathbb{E}}^{\underline{\Gamma}}(h)$
\item If $\vert \underline{\mathbb{E}}^{\overline{\Gamma}}(h) - \underline{\mathbb{E}}^{\underline{\Gamma}}(h) \vert \leq \epsilon$, stop, else increase $M$ and return to step 2.
\end{enumerate}
Note that the distance between two consecutive $\gamma_i,\gamma_{i+1}$ does not have to be constant. If $\underline{\Gamma}$ is inconsistent, an alternative is to use one of the two random sets $\Gamma_1,\Gamma_2$ such that for $i=1,\ldots,M$
$${\Gamma}_1 := \left\{ \begin{array}{c} A_{\gamma_{i,1}}=[a_{\ast\gamma_{i-1}},a_{\gamma_{i-1}}^{\ast}], 
 \\ m(A_{{\gamma_{i,1}}})=\gamma_i - \gamma_{i-1}, \end{array} \right. \quad {\Gamma}_2 := \left\{ \begin{array}{c}  A_{\gamma_{i,2}}=[a_{\ast\gamma_{i}},a_{\gamma_{i}}^{\ast}], 
 \\ m(A_{{\gamma_{i,2}}})=\gamma_i - \gamma_{i-1}. \end{array} \right.$$
The corresponding approximations read, for $j=1,2$,
$$ \underline{\mathbb{E}}^{\Gamma_j}(h)= \sum_{i=1}^M m(A_{\gamma_{i,j}}) \inf_{x \in A_{\gamma_{i,j}}} h(x) \textrm{ and } \overline{\mathbb{E}}^{\Gamma_j}(h)= \sum_{i=1}^M m(A_{\gamma_{i,j}}) \sup_{x \in A_{\gamma_{i,j}}} h(x).$$
Compared to $\underline{\Gamma}$,  $\Gamma_1,\Gamma_2$ have the advantage to always be consistent, but the obtained approximations can either outer- or inner-approximate the exact values, even if they converge towards it as $M$ increases.

\subsection{Conditional lower/upper expectations}

Another quite common problem when dealing with imprecise probabilities is the procedure of conditioning and the computations of associated lower/upper conditional expectations. Suppose that we observe an event $B=[b_{0},b_{1}]$. Then the lower and upper
conditional expectations, given the p-box $\pbox$ and under condition of $B$, can be determined as follows:%
\begin{align*}
\underline{\mathbb{E}}(h|B)  &  =\inf_{\underline{F}\leq F\leq\overline{F}%
}\frac{\int_{\mathbb{R}}h(x)I_{B}(x)\mathrm{d}F}{\int_{\mathbb{R}}%
I_{B}(x)\mathrm{d}F},\\
\overline{\mathbb{E}}(h|B)  &  =\sup_{\underline{F}\leq F\leq\overline{F}%
}\frac{\int_{\mathbb{R}}h(x)I_{B}(x)\mathrm{d}F}{\int_{\mathbb{R}}%
I_{B}(x)\mathrm{d}F}.
\end{align*}
The above formulas are equivalent to applying Bayes formula to every probability measure inside $\setpb$, and then retrieving the optimal bounds. Other generalisations of Bayes formula to imprecise probabilistic framework exist~\cite{DuboisPrade92c,Walley96}, but we will restrict ourselves to the above solution, as it is by far the most used within frameworks using lower/upper expectation bounds. Also, we assume that $B$ is large enough (or the two distributions $\pbox$ close enough) so that $\underline{F}(b_1) > \overline{F}(b_0)$. This is equivalent to require $\underline{P}(B) > 0$, thus avoiding conditioning on an event of probability $0$. Indeed, there are still some discussions about what should be done in presence of such events (see Miranda~\cite{Miranda08} for an introductory discussion and Cozman~\cite{Cozman02} for possible numerical solutions).

Similarly to unconditional expectations, the above problems can numerically be solved by
approximating the probability distribution function $F$ by a set of
$N$ points $F(x_{i})$, $i=1,...,N$, and by writing linear-fractional
optimization problems\footnote{Problems where the objective function is a fraction of two linear functions and constraints are linear.} and then associated linear programming problems. Problems
 mentioned for the unconditional case can again
occur. The next proposition indicates that previous results can be used to provide a more attractive formulation of $\underline{\mathbb{E}}(h|B),\overline{\mathbb{E}}(h|B)$.

%Figure~\ref{fig:NonMonF1} illustrates a potential optimal
%distribution $F$ for which upper conditional expectation is reached
%(under the condition $B=[1,8]$) when $h$ has one maximum.

\begin{proposition}
\label{prop:condexp} Given a p-box $\pbox$, a function $h(x)$ and an event $B$, the upper and lower conditional expectations
of $h(X)$
on $[\underline{F},\overline{F}]$ after observing the event $B$ can be written%
\begin{align}
\overline{\mathbb{E}}(h|B)=\sup_{\substack{\underline{F}(b_{0})\leq\alpha
\leq\overline{F}(b_{0})\\\underline{F}(b_{1})\leq\beta\leq\overline{F}(b_{1}%
)}}\frac{1}{\beta-\alpha}\Psi(\alpha,\beta), \label{eq:condexp-up}\\
\underline{\mathbb{E}}(h|B)=\inf_{\substack{\underline{F}(b_{0})\leq\alpha
\leq\overline{F}(b_{0})\\\underline{F}(b_{1})\leq\beta\leq\overline{F}(b_{1}%
)}}\frac{1}{\beta-\alpha}\Phi(\alpha,\beta), \label{eq:condexp-low}
\end{align}
with
\begin{align*}
\Psi(\alpha,\beta)  &  =\int_{\alpha}^{\beta}\sup_{x\in A_{\gamma} \cap B}h(x)\mathrm{d}\gamma.\\
\Phi(\alpha,\beta)  &  =\int_{\alpha}^{\beta}\inf_{x\in A_{\gamma} \cap B}h(x)\mathrm{d}\gamma.
\end{align*}
\end{proposition}

\begin{proof}
[\textbf{General proof}]We consider only upper expectation. We do
not know how the extremizing distribution function behaves outside the
interval $B$. Therefore, we suppose that the value of the extremizing
distribution function at point
$b_{0}$ is $F(b_{0})=\alpha\in\lbrack\underline{F}(b_{0}),\overline{F}%
(b_{0})]$ and its value at point $b_{1}$ is
$F(b_{1})=\beta\in\lbrack \underline{F}(b_{1}),\overline{F}(b_{1})]$
(see Fig. \ref{fig:NonMonF1}). Then
there holds%
\[
\int_{\mathbb{R}}I_{B}(x)\mathrm{d}F(x)=\beta-\alpha.
\]
Hence, we can write
\begin{align}
&  \ \overline{\mathbb{E}}(h|B) =\sup_{\substack{\underline{F}(b_{0}%
)\leq\alpha\leq\overline{F}(b_{0}) \\\underline{F}(b_{1})\leq\beta
\leq\overline{F}(b_{1}) \\\underline{F}\leq F\leq\overline{F}}}\frac{1}%
{\beta-\alpha}\int_{\mathbb{R}}h(x)I_{B}(x)\mathrm{d}F(x)\nonumber\\
&
=\sup_{\substack{\underline{F}(b_{0})\leq\alpha\leq\overline{F}(b_{0})
\\\underline{F}(b_{1})\leq\beta\leq\overline{F}(b_{1})}}\frac{1}{\beta-\alpha
}\left(  \sup_{\substack{\underline{F}\leq F\leq\overline{F} \\F(b_{0}%
)=\alpha\\F(b_{1})=\beta}}\int_{\mathbb{R}}h(x)I_{B}(x)\mathrm{d}F(x)\right)
\nonumber\\
&
=\sup_{\substack{\underline{F}(b_{0})\leq\alpha\leq\overline{F}(b_{0})
\\\underline{F}(b_{1})\leq\beta\leq\overline{F}(b_{1})}}\frac{1}{\beta-\alpha
}\int_{\alpha}^{\beta}\sup_{x\in A_{\gamma} \cap B}h(x)\mathrm{d}\gamma.
\label{eq:condupper}%
\end{align}
By using the results obtained for the unconditional
upper expectation, we can see that the integrand is equal to
$\Psi(\alpha,\beta)$. The lower expectation is similarly proved.
\end{proof}

As value $\beta-\alpha$
increases in Equations \eqref{eq:condexp-up}-\eqref{eq:condexp-low}, so do the numerator and denominator, thus playing
opposite role in the evolution of the objective function. Hence,
in order to compute the upper (resp. lower) conditional expectation, one has to
find the values $\beta$ and $\alpha$ such that any increase (decrease)
in the value $\beta-\alpha$ is greater (resp. lower) than the
corresponding increase (resp. decrease) in $\Psi(\alpha,\beta)$ ($\Phi(\alpha,\beta)$).

A crude algorithm to approximate the solution would be to samples different values $\alpha \in [\underline{F}(b_0),\overline{F}(b_{0})]$ and $\beta \in [\underline{F}(b_1),\overline{F}(b_{1})]$, evaluating Equations \eqref{eq:condexp-up}-\eqref{eq:condexp-low} for all combination $[\alpha,\beta]$ and retaining the highest obtained value (note that we can have $\overline{F}(b_{0})
\geq\underline{F}(b_{1})$, hence the need to make sure by adding constraint that $[\alpha,\beta]$ is not void). 

Another interesting point to note is that the proof takes advantage
of both views, since the idea to use levels $\alpha$ and $\beta$
comes from fractional linear programming, while the final equation
(\ref{eq:condupper}) can be elegantly formulated by using the random
set view.

In any cases (lower/upper and conditional/unconditional expectations), it is obvious that the extremizing probability distribution $F$ providing the minimum (resp. maximum) expectation of $h$
depends on the form of the function $h$. If this form follows some
typical cases, efficient solutions can be found to compute lower
(resp. upper) expectations. The simplest examples (for which solutions are
well known) of such typical cases are monotone functions.

\section{The simple case of monotone functions}
\label{sec:mon-func}

We first consider the case where $h$ is a monotone function that is non-decreasing
(resp. non-increasing) in $\mathbb{R}$. We will also introduce the running example used throughout the paper.

\subsection{Unconditional expectations}

In the case of a monotone non-decreasing (resp. non-increasing) function, existing results~\cite{Walley96} tell us that we have:
\begin{align}
&
\underline{\mathbb{E}}(h)=\int_{\mathbb{R}}\!\!h(x)\mathrm{d}\overline
{F}\;\left(  \underline{\mathbb{E}}(h)=\int_{\mathbb{R}}\!\!h(x)\mathrm{d}%
\underline{F}\right)  ,\label{NonMonF33}\\
&
\overline{\mathbb{E}}(h)=\int_{\mathbb{R}}\!\!h(x)\mathrm{d}\underline
{F}\;\left(  \overline{\mathbb{E}}(h)=\int_{\mathbb{R}}\!\!h(x)\mathrm{d}%
\overline{F}\right)  , \label{NonMonF34}%
\end{align}
and we see from (\ref{NonMonF33})-(\ref{NonMonF34}) that lower and
upper expectations are completely determined by bounding
distributions
$\underline{F}$ and $\overline{F}$. Using equations (\ref{eq:Lowerprevleb}%
)-(\ref{eq:upperprevleb}), we get the following formulas
\begin{align}
\underline{\mathbb{E}}(h)= \int_{0}^{1} h(a_{\ast\gamma})d\gamma
\;\left(
\underline{\mathbb{E}}(h)=\int_{0}^{1} h(a_{\gamma}^{\ast
})d\gamma\right)   &  ,\label{eq:Lowerprevmon}\\
\overline{\mathbb{E}}(h)=\int_{0}^{1}h(a_{\gamma}^{\ast}%
)d\gamma\;\left(
\overline{\mathbb{E}}(h)=\int_{0}^{1} h(a_{\ast
\gamma})d\gamma\right)   &  , \label{eq:upperprevmon}%
\end{align}
which are the counterparts of equations
(\ref{NonMonF33})-(\ref{NonMonF34}). Here, expectations are totally
determined by extreme values of the mappings. When $h$ is
non-monotone, equations (\ref{NonMonF33})-(\ref{eq:upperprevmon}) only
provide inner approximations of
$\underline{\mathbb{E}}(h)$,$\overline{\mathbb{E}}(h)$. When using numerical procedures over monotone functions, there appears to be no specific sampling strategies of values that would allow for faster convergence. 

We now introduce the example that will illustrate our results all along the paper.

\begin{example} \label{exmp:mono-uncon} Assume that we have to estimate the loss incurred by the failure of a unit of some industrial item. Suppose that this loss is the function of time $h(x)=20-x$,
and it is known that the unit time to failure is governed by a distribution
whose bounds are exponential distributions with a failure rate $0.2$ and
$0.5$ (note that only the bounds are of exponential nature). $h$ is decreasing and can, for example, model the fact that the later the unit fails, the less it costs to replace it. Let us compute the expected losses as the expectation of $h$. The lower
and upper distribution functions of the unit time to failure are
$1-\exp(-0.2x)$ and $1-\exp(-0.5x)$, respectively. Hence
\[
\overline{\mathbb{E}}(h)=\int_{0}^{\infty}(20-x)\mathrm{d}(1-\exp(-0.5x))=\int
_{0}^{\infty}(20-x)0.5e^{-0.5x}\mathrm{d}x=18,
\]%
\[
\underline{\mathbb{E}}(h)=\int_{0}^{\infty}(20-x)\mathrm{d}(1-\exp
(-0.2x))=\int_{0}^{\infty}(20-x)0.2e^{-0.2x}\mathrm{d}x=15.
\]
Finally, we obtain that the expected losses are in the interval $[15,\ 18]$.

Let us use the random set approach. Since $\overline{F}^{-1}(\gamma
)=-2\ln(1-\gamma)=a_{\gamma}^{\ast}$ and $\underline{F}^{-1}(\gamma
)=-5\ln(1-\gamma)=a_{\ast\gamma}$, then%
\[
\overline{\mathbb{E}}(h)=\int_{0}^{1}(20+2\ln(1-\gamma))\mathrm{d}\gamma=18,
\]%
\[
\underline{\mathbb{E}}(h)=\int_{0}^{1}(20+5\ln(1-\gamma))\mathrm{d}\gamma=15.
\]
We get the same values of the lower and upper expectations of $h$. 
\end{example}

\subsection{Conditional expectations}

We now consider that we want to know the lower and upper expectations in the case where event $B=[b_0,b_1]$ occurs. That is, we want to compute Equations~\eqref{eq:condexp-up},~\eqref{eq:condexp-low} for a monotone $h$. Lower and upper expectations are then given by the following proposition.

\begin{proposition}
 Given a p-box $\pbox$, a monotone function $h(x)$ and an event $B$, the upper and lower conditional expectation of $h(X)$
on $[\underline{F},\overline{F}]$ after observing the event $B$ can be written
\begin{align*}
\overline{\mathbb{E}}(h|B) & =\sup_{\substack{\underline{F}(b_{0})\leq
\alpha\leq\overline{F}(b_{0})\\\underline{F}(b_{1})\leq\beta\leq\overline
{F}(b_{1})}}\frac{1}{\beta-\alpha}\int_{\alpha}^{\beta}\sup_{x\in A_{\gamma} \cap B} h(x) \mathrm{d}\gamma \\
& =\frac{1}{\overline{F}(b_{1})-\overline{F}(b_{0}%
)}\left(\int_{\underline{F}^{-1}(\overline{F}(b_{0}))}^{b_{1}}h(x)\mathrm{d}%
\underline{F}(x)+h(b_{1})\left(  \overline{F}(b_{1})-\underline{F}%
(b_{1})\right) \right), \\
\underline{\mathbb{E}}(h|B) & =\inf_{\substack{\underline{F}(b_{0})\leq
\alpha\leq\overline{F}(b_{0})\\\underline{F}(b_{1})\leq\beta\leq\overline
{F}(b_{1})}}\frac{1}{\beta-\alpha}\int_{\alpha}^{\beta}\inf_{x\in A_{\gamma} \cap B} h(x)\mathrm{d}\gamma \\
& =\frac{1}{\underline{F}(b_{1})-\underline{F}(b_{0}%
)}\left( h(b_{0})\left(  \overline{F}(b_{0})-\underline{F}%
(b_{0})\right) + \int_{b_{0}}^{\overline{F}^{-1}(\underline{F}(b_1))}h(x)\mathrm{d}%
\overline{F}(x) \right)
\end{align*}
if $h$ is non-decreasing and
\begin{align*}
\overline{\mathbb{E}}(h|B) & =\frac{1}{\underline{F}(b_{1})-\underline{F}(b_{0}%
)}\left(h(b_{0})\left(  \overline{F}(b_{0})-\underline{F}%
(b_{0})\right) + \int_{b_{0}}^{\overline{F}^{-1}(\underline{F}(b_1))}h(x)\mathrm{d}%
\overline{F}(x) \right), \\
\underline{\mathbb{E}}(h|B) &  =\frac{1}{\overline{F}(b_{1})-\overline{F}(b_{0}%
)}\left(\int_{\underline{F}^{-1}(\overline{F}(b_{0}))}^{b_{1}}h(x)\mathrm{d}%
\underline{F}(x)+h(b_{1})\left(  \overline{F}(b_{1})-\underline{F}%
(b_{1})\right) \right), \end{align*}
if $h$ is non-increasing.
\end{proposition}

\begin{proof}
We will only prove the upper expectation for non-decreasing function $h$. Lower expectation can be derived likewise, and the case of non-increasing functions is then obtained by using duality between lower and upper expectations.

When $h$ is non-decreasing, we know that $\sup_{x\in A_{\gamma} \cap B} h(x)$ is a non-decreasing function of $\gamma$ that coincides with $\underline{F}^{-1}$. Using the integral mean value theorem, we know that there exists some $z\in[b_0,b_1]$ such that $\overline {\mathbb{E}}(h|B)=h(z)$, whatever the choice of $\alpha,\beta$. For maximizing
$\overline{\mathbb{E}}(h|B)$, values $\alpha,\beta$ should be chosen so that the retained values $z$ and $h(z)$ (coinciding with $\underline{F}^{-1}$) are as high as possible. As $h$ is non-decreasing, this corresponds to values $\alpha=\overline{F}%
(b_{0})$, $\beta=\overline{F}(b_{1})$, which settles the denominator of the objective function. We then have 
$$\int_{\alpha}^{\beta}\sup_{x\in A_{\gamma} \cap B} h(x) \mathrm{d}\gamma=\int_{\underline{F}^{-1}(\overline{F}(b_{0}))}^{b_{1}}h(x)\mathrm{d}%
\underline{F}(x)+h(b_{1})\left(  \overline{F}(b_{1})-\underline{F}%
(b_{1})\right),$$ because for values $\gamma \in [\overline{F}(b_{0}),\underline{F}(b_{1})]$, supremum of $h(x)$ on $A_{\gamma} \cap B$ is obtained for $x=\underline{F}^{-1}(\gamma)$, while for  $\gamma \in [\underline{F}(b_{1}),\overline{F}(b_{1})]$, supremum of $h(x)=b_1$.
\end{proof}

\begin{center}
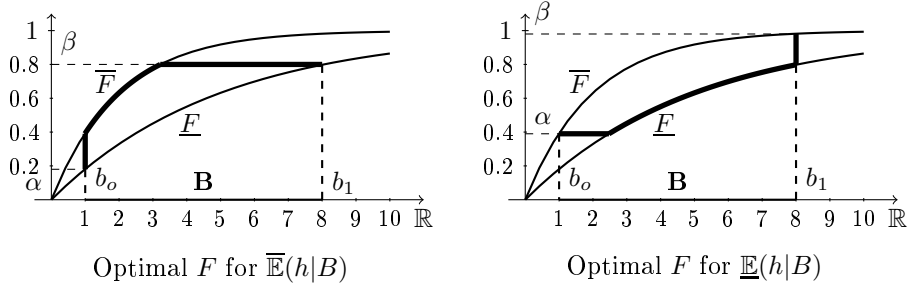
\begin{figure}[htb]
\centering
\begin{tikzpicture}[scale=0.45]

\draw[->] (-0.5,0) -- (11,0) node[below] {$\mathbb{R}$};
\draw[->] (0,0) -- (0,5.5);
\draw (-0.1,5) node[left] {$1$} -- (0.1,5);
\draw[domain=0:10,thick]  plot[id=exp1] function{5*(1-exp(-0.2*x))} ;
\node at (1.6,3.5) {$\udf$};
%\draw[domain=-1:5,ultra thick,dashed]  plot[id=norm] function{3*norm((x-1)*2)};
\draw[domain=0:10,thick]  plot[id=exp2] function{5*(1-exp(-0.5*x))};
\draw[dashed] (0,0.18*5) node[below left] {$\alpha$} -- (1,0.18*5);
\draw[dashed] (0,0.8*5) node[above right] {$\beta$} -- (8,0.8*5);
\draw[dashed,thick] (1,0.18*5) -- (1,0) node[above right] {$b_o$};
\draw[thick]  (1,0) -- (8,0) node[midway,above] {\textbf{B}};
\draw[dashed,thick] (8,0)  -- (8,0.8*5) node[very near start,right] {$b_1$};
\foreach \x in {1,...,10} {\draw (\x , 0.05) -- (\x , -0.05) node[below] {\small \x};}
\foreach \y in {0.2,0.4,0.6,0.8} {\draw (0.05 , \y*5) -- (-0.05 , \y*5) node[left] {\small \y};}
\draw[line width=2pt] (1,0.18*5) -- (1,0.39*5);
\draw[domain=0.99:3.218, line width=2pt]  plot[id=exp3] function{5*(1-exp(-0.5*x))} ;
\draw[line width=2pt] (3.218,0.8*5) -- (8,0.8*5);
%\draw[thick,dotted] (-0.1,1.5) node[left] {$\gamma$} -- (4,1.5) ; 
%\draw[thick,dotted] (1.5,1.5) -- (1.5,-0.2) node[below] {$a_{\ast\gamma}$} ;
%\draw[thick,dotted] (3,1.5) -- (3,-0.2) node[below] {$a^{\ast}_\gamma$};
%\draw[very thick] (1.5,1.5) -- (3,1.5) node[midway,above] { \Large $A_{\gamma}$};
\node at (4.1,2.2) {$\ldf$};
\node at (5,-2) {Optimal $\df$ for $\overline{\mathbb{E}}(h|B)$ };

\begin{scope}[xshift=14cm]
\draw[->] (-0.5,0) -- (11,0) node[below] {$\mathbb{R}$};
\draw[->] (0,0) -- (0,5.5);
\draw (-0.1,5) node[left] {$1$} -- (0.1,5);
\draw[domain=0:10,thick]  plot[id=exp4] function{5*(1-exp(-0.2*x))} ;
\node at (1.6,3.5) {$\udf$};
%\draw[domain=-1:5,ultra thick,dashed]  plot[id=norm] function{3*norm((x-1)*2)};
\draw[domain=0:10,thick]  plot[id=exp5] function{5*(1-exp(-0.5*x))};
\draw[dashed] (0,0.39*5) node[above right] {$\alpha$} -- (2.47,0.39*5);
\draw[dashed] (0,0.98*5) node[above right] {$\beta$} -- (8,0.98*5);
\draw[dashed,thick] (1,0.39*5) -- (1,0) node[above right] {$b_o$};
\draw[thick]  (1,0) -- (8,0) node[midway,above] {\textbf{B}};
\draw[dashed,thick] (8,0)  -- (8,0.98*5) node[very near start,right] {$b_1$};
\foreach \x in {1,...,10} {\draw (\x , 0.05) -- (\x , -0.05) node[below] {\small \x};}
\foreach \y in {0.2,0.4,0.6,0.8} {\draw (0.05 , \y*5) -- (-0.05 , \y*5) node[left] {\small \y};}
\draw[line width=2pt] (8,0.8*5) -- (8,0.98*5);
\draw[domain=2.471:8, line width=2pt]  plot[id=exp6] function{5*(1-exp(-0.2*x))} ;
\draw[line width=2pt] (1,0.39*5) -- (2.471,0.39*5);
%\draw[thick,dotted] (-0.1,1.5) node[left] {$\gamma$} -- (4,1.5) ; 
%\draw[thick,dotted] (1.5,1.5) -- (1.5,-0.2) node[below] {$a_{\ast\gamma}$} ;
%\draw[thick,dotted] (3,1.5) -- (3,-0.2) node[below] {$a^{\ast}_\gamma$};
%\draw[very thick] (1.5,1.5) -- (3,1.5) node[midway,above] { \Large $A_{\gamma}$};
\node at (4.1,2.2) {$\ldf$};
\node at (5,-2) {Optimal $\df$ for $\underline{\mathbb{E}}(h|B)$ };
\end{scope}
\end{tikzpicture}
\caption{Conditional expectations with monotone non-increasing functions}%
\label{fig:pboxcondmono}%
\end{figure}
\end{center}

\begin{example} \label{exmp:mono-cond} We consider the same p-box $\pbox$ and function $h$ as in Example~\ref{exmp:mono-uncon}, but now we consider that we want to know the incurred loss in case $x \in B=[1,8]$, that is the failure is supposed to happen between 1 and 8 units of time. We have
\[
\underline{F}(b_{0})=1-\exp(-0.2\cdot1)=0.18, \quad \overline{F}(b_{0})=1-\exp(-0.5\cdot1)=\allowbreak0.39,
\]%
\[
\underline{F}(b_{1})=1-\exp(-0.2\cdot8)=\allowbreak0.8, \quad
\overline{F}(b_{1})=1-\exp(-0.5\cdot8)=\allowbreak0.98,
\]
and we get 
\begin{align*}\overline{\mathbb{E}}(h|B) & =\frac{1}{0.8-0.18} \left((20-1)\left(  0.39 - 0.18 \right) + \int_{1}^{\overline{F}^{-1}(0.8)}(20-x)0.5e^{-0.5x}\mathrm{d}x \right) \\
& = 18.298, \\ 
\underline{\mathbb{E}}(h|B) & =\frac{1}{0.98-0.39} \left((20-8)\left(  0.98 - 0.8 \right) + \int_{\underline{F}^{-1}(0.39)}^{8}(20-x)0.2e^{-0.2x}\mathrm{d}x \right) \\
& = 14.219.
\end{align*}
Note that, if we compare above values with those of Example~\ref{exmp:mono-uncon}, we have $[\underline{\mathbb{E}}(h),\overline{\mathbb{E}}(h)] \subset [\underline{\mathbb{E}}(h|B),\overline{\mathbb{E}}(h|B)]$.
\end{example}

The above results indicate that, when $h$ is monotone, computing lower/upper expectations exactly remains easy. Also, when using numerical methods, they provide insight as to how values should be sampled. For example, when computing upper conditional expectation by linear programming, values only need to be sampled in $[b_0,\udf^{-1}(b_1)]$, and $b_0$ should be among the sampled values, since an important probability mass is concentrated at this value (see Fig.~\ref{fig:pboxcondmono}). When using random set approach and discretizing the unit interval $[0,1]$, one should take $\gamma_1=\ldf{b_0}$ and $\gamma_2=\udf(b_0)$, and not consider finer discretization of this interval, as this would not increase the precision. As we shall see, similar results can be derived for more complex cases.

\section{Function with one maximum \label{sec:maxunivar}}

In this section, we study the case where the function $h$ has one
maximum at point $a$, i.e. $h$ is increasing (resp. decreasing) in
$(-\infty,a]$ (resp. $[a,\infty )$). The case of $h$ having one minimum follows by considering the function $
-h$ and the duality between lower and upper expectations.

\subsection{Unconditional expectations}

As for monotone $h$, we first study the case of unconditional expectations. Before giving the main result, we show the next lemma that will be useful in subsequent proofs.

\begin{lemma}\label{lem:equationsol}
Given a p-box $\pbox$ and a continuous function $h(x)$ with one maximum at $x=a$, there is always a solution $\gamma \in [\ldf(a),\udf(a)]$ to the following equation
\begin{equation}
h\left(  \overline{F}^{-1}(\gamma)\right)  =h\left(  \underline{F}^{-1}%
(\gamma)\right)  . %
\end{equation}
\end{lemma}

\begin{proof}
let us
consider the function%
\[
\varphi\left(  \alpha\right)  =h\left(  \overline{F}^{-1}\left(
\alpha\right)  \right)  -h\left(  \underline{F}^{-1}\left(
\alpha\right) \right)  ,
\]
\newline which, being a substraction of two continuous functions (by
supposition), is continuous. Since the function $h$ has its maximum
at point $x=a$, then, by taking $\alpha=\underline{F}\left( a\right)
$, we get the
inequality%
\[
\varphi\left(  \gamma\right)  =h\left(  \overline{F}^{-1}\left(
\underline {F}\left(  a\right)  \right)  \right)  -h\left(  a\right)
\leq0
\]
and, by taking $\gamma=\overline{F}\left(  a\right)  $, we get the inequality%
\[
\varphi\left(  \gamma\right)  =h\left(  a\right)  -h\left(  \underline{F}%
^{-1}\left(  \overline{F}\left(  a\right)  \right)  \right)  \geq0.
\]
Consequently, there exists $\gamma$ in the interval $\left(
\underline {F}\left(  a\right)  ,\overline{F}\left(  a\right)
\right)  $ such that $\varphi\left(  \gamma\right)  =0$ (since
$\varphi$ is continuous). 
\end{proof}

The next proposition shows that, as for monotone $h$, the fact of knowing that $h$ has one maximum in $x=a$ allows us to derive closed-form expressions of lower and upper expectations. The results of the proposition are illustrated in Fig.~\ref{fig:onemax-uncon}.

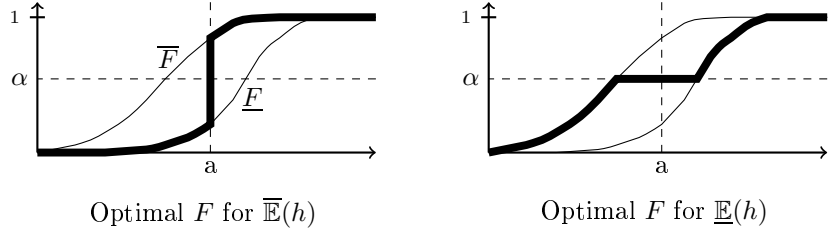
\begin{figure}
\begin{center}
\begin{tikzpicture}

\draw[thick, ->] (0,0) -- (45mm,0mm) ; \draw[thick]
(-1mm,18mm) node[left] {$\scriptstyle 1$} -- (1mm,18mm) ;
\draw[thick, ->] (0,0) -- (0mm,20mm) ;

%\draw[step=0.5mm,gray,very thin] (0,0) grid  (45mm,25mm) ;

%\draw (0,0) .. controls (2.5mm,1mm) .. (5.5mm,2mm) .. controls
%(6.25mm,2.50mm) .. (8mm,3.5mm) .. controls (8.75mm,3.75mm) ..
%(9.5mm,4.5mm) .. controls (11.5mm,5.25mm) .. (13.5mm,7mm) --
%(15mm,9.5mm) .. controls (16.75mm,11.25mm) .. (18.5mm,12.5mm) ..
%controls (20mm,14mm) .. (21.5mm,15mm) .. controls (24mm,16.5mm) ..
%(26.5mm,17.5mm) .. controls (27.75mm,18.5mm) .. (29mm,19mm) ..
%controls (31mm,19.75mm) .. (33mm,20mm) -- (35.5mm,20mm) --
%(40mm,20mm) -- (45mm,20mm) ;

\draw (0,0) -- (5mm,1mm) .. controls (7.5mm,1.75mm) ..
(10.75mm,3.75mm) .. controls (12.5mm,5mm) .. (14.5mm,7mm) --
(17mm,9.8mm) -- (20mm,12.75mm) node[left] {$\overline{F}$} -- (23mm,15.25mm) -- (26mm,17.2mm) ..
controls (27.5mm,17.8mm) .. (32.2mm,18mm) -- (37mm,18mm) --
(45mm,18mm) ;

\draw (0,0) -- (9mm,0mm) -- (14.5mm,0.4mm) .. controls
(15.75mm,0.6mm) .. (17mm,1mm) -- (20mm,2mm) .. controls
(21.5mm,2.75mm) .. (23mm,3.75mm) -- (26mm,7mm) node[right] {$\underline{F}$} -- (27.7mm,9.8mm) --
(29.5mm,12.75mm) .. controls (30.8mm,14.3mm) .. (32.2mm,15.25mm) ..
controls (34.6mm,17.2mm) .. (37mm,18mm) -- (45mm,18mm) ;

%\draw (0,0) -- (4mm,0mm) .. controls (6mm,0.4mm) .. (8mm,1mm) ..
%controls (8.75mm,1.4mm) .. (9.5mm,2mm) .. controls (11.5mm,3mm) ..
%(13.5mm,3.5mm) .. controls (14.25mm,3.8mm) .. (15mm,4.5mm) --
%(18.5mm,7mm) -- (19.7mm,8mm) -- (21.5mm,9.5mm) -- (26mm,12.5mm) --
%(29mm,14.5mm) -- (33mm,17.5mm) -- (35.5mm,19mm) .. controls
%(37.75mm,19.80mm) .. (40mm,20mm) -- (45mm,20mm) ;

\draw[line width=3pt] (0,0) -- (9mm,0mm) -- (14.5mm,0.4mm) ..
controls (15.75mm,0.6mm) .. (17mm,1mm) -- (20mm,2mm) .. controls
(21.5mm,2.75mm) .. (23mm,3.75mm) -- (23mm,15.25mm) -- (26mm,17.2mm)
.. controls (27.5mm,17.8mm) .. (32.2mm,18mm) -- (37mm,18mm) --
(45mm,18mm) ;

\draw[dashed] (23mm,20mm) -- (23mm,0mm) node[below] {a} ;

\draw[dashed] (45mm,9.8mm) -- (0mm,9.8mm) node[left] {$\alpha$} ;

\begin{scope}[xshift=6cm]

\draw[thick, ->] (0,0) -- (45mm,0mm) ; \draw[thick]
(-1mm,18mm) node[left] {$\scriptstyle 1$} -- (1mm,18mm) ;
\draw[thick, ->] (0,0) -- (0mm,20mm) ;

%\draw[step=0.5mm,gray,very thin] (0,0) grid  (45mm,25mm) ;

%\draw (0,0) .. controls (2.5mm,1mm) .. (5.5mm,2mm) .. controls
%(6.25mm,2.50mm) .. (8mm,3.5mm) .. controls (8.75mm,3.75mm) ..
%(9.5mm,4.5mm) .. controls (11.5mm,5.25mm) .. (13.5mm,7mm) --
%(15mm,9.5mm) .. controls (16.75mm,11.25mm) .. (18.5mm,12.5mm) ..
%controls (20mm,14mm) .. (21.5mm,15mm) .. controls (24mm,16.5mm) ..
%(26.5mm,17.5mm) .. controls (27.75mm,18.5mm) .. (29mm,19mm) ..
%controls (31mm,19.75mm) .. (33mm,20mm) -- (35.5mm,20mm) --
%(40mm,20mm) -- (45mm,20mm) ;

\draw (0,0) -- (5mm,1mm) .. controls (7.5mm,1.75mm) ..
(10.75mm,3.75mm) .. controls (12.5mm,5mm) .. (14.5mm,7mm) --
(17mm,9.8mm) -- (20mm,12.75mm) -- (23mm,15.25mm) -- (26mm,17.2mm) ..
controls (27.5mm,17.8mm) .. (32.2mm,18mm) -- (37mm,18mm) --
(45mm,18mm) ;

\draw (0,0) -- (9mm,0mm) -- (14.5mm,0.4mm) .. controls
(15.75mm,0.6mm) .. (17mm,1mm) -- (20mm,2mm) .. controls
(21.5mm,2.75mm) .. (23mm,3.75mm) -- (26mm,7mm) -- (27.7mm,9.8mm) --
(29.5mm,12.75mm) .. controls (30.8mm,14.3mm) .. (32.2mm,15.25mm) ..
controls (34.6mm,17.2mm) .. (37mm,18mm) -- (45mm,18mm) ;

\draw[line width=3pt] (0,0) -- (5mm,1mm) .. controls (7.5mm,1.75mm)
.. (10.75mm,3.75mm) .. controls (12.5mm,5mm) .. (14.5mm,7mm) --
(17mm,9.8mm) -- (27.7mm,9.8mm) -- (29.5mm,12.75mm) .. controls
(30.8mm,14.3mm) .. (32.2mm,15.25mm) .. controls (34.6mm,17.2mm) ..
(37mm,18mm) -- (45mm,18mm) ;

\draw[dashed] (23mm,20mm) -- (23mm,0mm) node[below] {a} ;

\draw[dashed] (45mm,9.8mm) -- (0mm,9.8mm) node[left] {$\alpha$} ;

\end{scope}

\draw (2.2,-0.8) node {Optimal $F$ for
$\overline{\mathbb{E}}(h)$} ;
\begin{scope}[xshift=6cm]
\draw (2.2,-0.8) node {Optimal $F$ for
$\underline{\mathbb{E}}(h)$ } ;
\end{scope}
\end{tikzpicture}
\end{center} 
\caption{Optimal distributions $F$ with unimodal $h$}
\label{fig:onemax-uncon}
\end{figure}

\begin{proposition}
\label{pr:NonMonF1}If the function $h$ has one maximum at point
$a\in \mathbb{R}$, then the upper and lower expectations of $h(X)$
on $[\underline
{F},\overline{F}]$ are%
\begin{equation}
\overline{\mathbb{E}}(h)
=\int\limits_{-\infty}^{a}h(x)\mathrm{d}\underline{F}
+h(a)\left[ \overline{F}(a)-\underline {F}(a)\right]
  +\int\limits_{a}%
^{\infty}h(x)\mathrm{d}\overline{F}, \label{NonMonF35}%
\end{equation}%
\begin{equation}
\underline{\mathbb{E}}(h)=\left[  \int\limits_{-\infty}^{\overline{F}%
^{-1}(\alpha)}h(x)\mathrm{d}\overline{F}+\int\limits_{\underline{F}%
^{-1}(\alpha)}^{\infty}h(x)\mathrm{d}\underline{F}\right]  , \label{NonMonF36}%
\end{equation}
or, equivalently
\begin{align}
\overline{\mathbb{E}}(h)  &  =\int\limits_{0}^{\underline{F}%
(a)}h(a_{_{\gamma}}^{\ast})d\gamma+[{\overline{F}(a)}-{\underline{F}%
(a)}]h(a)
+\int\limits_{\overline{F}(a)}^{1}h(a_{_{\ast\gamma}
})d\gamma \label{eq:unilebupper}%
\end{align}%
\begin{equation}
\underline{\mathbb{E}}(h)=\int\limits_{0}^{\alpha}h(a_{\ast\gamma}%
)d\gamma+\int\limits_{\alpha}^{1}h(a_{\gamma}^{\ast})d\gamma,
\label{eq:unileblower}%
\end{equation}
where $\alpha$ is the solution of equation
\begin{equation}
h\left(  \overline{F}^{-1}(\alpha)\right)  =h\left(  \underline{F}^{-1}%
(\alpha)\right)  . \label{NonMonF38}%
\end{equation}
such that $\alpha \in [\ldf(a),\udf(a)]$.
\end{proposition}

\begin{proof}
[\textbf{Proof using linear programming}] We assume that the function
$h\left( x\right)  $ is differentiable in $\mathbb{R}$ and has a
finite value as $x\rightarrow\infty$. The lower and upper cumulative
probability functions $\underline{F}$ and $\overline{F}$ are also assumed to be
differentiable. We also consider the primal and dual problems considered in Section~\ref{sec:linprog-view} and recalled below.

\begin{equation*}
\begin{array}{c@{\hspace{0.5cm}}c@{\hspace{0.5cm}}c}
 \cline{1-1} \cline{3-3}
\textrm{\textbf{Primal problem:}} & &  \textrm{\textbf{Dual problem:}} \\
 \cline{1-1} \cline{3-3} \\ \textrm{Min.} \ \
\mathbf{v}=\int\limits_{-\infty}^{\infty}h\left(  x\right)  \rho\left(  x\right)
\mathrm{d}{x} & &
 \textrm{Max.} \ \
\mathbf{w}=c_{0}+\int\limits_{-\infty}^{\infty}\left(  -c\left(  t\right) \overline
{F}\left(  t\right)  +d\left(  t\right) \underline{F}\left(
t\right)
\right)  \mathrm{d}{t}\\  & & \\
\textrm{subject to} & & \textrm{subject to}\\ & & \\
\rho\left(  x\right)  \geq0, \int\limits\limits_{-\infty}^{\infty}\rho\left(
x\right)
\mathrm{d}{x}=1, & &
 c_{0}+\int\limits_{x}^{\infty}\left(  -c\left(  t\right)  +d\left(
t\right)
\right)  \mathrm{d}{t}\leq h\left(  x\right)  , \\
-\int\limits_{-\infty}^{x}\rho\left(  x\right) \mathrm{d}{x}\geq-\overline
{F}\left(  x\right)  ,& &
 c_{0}\in\mathbb{R}, c\left(  x\right)  \geq0, d\left(  x\right)  \geq0.\\
\int\limits_{-\infty}^{x}\rho\left(  x\right)
\mathrm{d}{x}\geq\underline{F}\left( x\right)  . & &\\ & &\\ \cline{1-1} \cline{3-3}
\end{array}
\end{equation*}

The proof of Equations (\ref{NonMonF35})-(\ref{NonMonF36}) and
(\ref{NonMonF38}) can be separated in three main steps:

\begin{enumerate}
\item We propose a feasible solution of the primal problem. 
\item We then consider the feasible solution of the dual problem corresponding
to the one proposed for the primal problem. 
\item We show that the two solutions coincide and, therefore, according to the
basic duality theorem of linear programming, these solutions are
optimal ones.
\end{enumerate}

First, we consider the primal problem. Let $a^{\prime}$ and
$a^{\prime\prime}$ be real values. The function
\[
\rho\left(  x\right)  =\left\{
\begin{array}
[c]{cc}%
\mathrm{d}\overline{F}\left(  x\right)  /\mathrm{d}x, & x<a^{\prime}\\
0, & a^{\prime}\leq x\leq a^{\prime\prime}\\
\mathrm{d}\underline{F}\left(  x\right)  /\mathrm{d}x, &
a^{\prime\prime}<x
\end{array}
\right.
\]
is a feasible solution to the primal problem if the following
conditions are
respected:%
\[
\int_{-\infty}^{\infty}\rho\left(  x\right)  \mathrm{d}{x}=1,
\]
which, given the above solution, can be rewritten
\[
\int_{-\infty}^{a^{\prime}}\mathrm{d}\overline{F}+\int_{a^{\prime\prime}%
}^{\infty}\mathrm{d}\underline{F}=1,
\]
which is equivalent to the equality
\begin{equation}
\overline{F}\left(  a^{\prime}\right)  =\underline{F}\left(
a^{\prime\prime
}\right)  .\label{prduco2}%
\end{equation}
We now interest ourselves in the dual problem. Let us first consider
the sole constraint
\begin{equation}
c_{0}+\int_{x}^{\infty}\left(  -c\left(  t\right)  +d\left( t\right)
\right)  \mathrm{d}{t}\leq h\left(  x\right)  ,\label{prduco1}%
\end{equation}
which is the equivalent of the primal constraint $\rho\left(
x\right)  \geq 0$. We then consider the following feasible solution
to the dual problem as $c_{0}=h\left(  \infty\right)  $,
\[
c\left(  x\right)  =\left\{
\begin{array}
[c]{l@{\hspace{0.2cm}}l}%
h^{\prime}\left(  x\right)  , & x<a^{\prime}\\
0, & x\geq a^{\prime}%
\end{array}
\right.  d\left(  x\right)  =\left\{
\begin{array}
[c]{l@{\hspace{0.2cm}}l}%
0, & x<a^{\prime\prime}\\
-h^{\prime}\left(  x\right)  , & x\geq a^{\prime\prime}%
\end{array}
\right.  .
\]
\newline The inequalities $c\left(  x\right)  \geq0$ and $d\left(  x\right)
\geq0$ are valid provided we have the inequalities $a^{\prime}\leq
a\leq a^{\prime\prime}$ (i.e. interval
$[a^{\prime},a^{\prime\prime}]$ encompasses maximum of $h$). By
integrating $c\left(  x\right)  $ and $d\left(  x\right)
$, we get the increasing function%
\[
C\left(  x\right)  =-\int_{x}^{\infty}c\left(  t\right)  \mathrm{d}%
{t}=\left\{
\begin{array}
[c]{ll}%
h\left(  x\right)  -h\left(  a^{\prime}\right)  , & x<a^{\prime}\\
0, & x\geq a^{\prime}%
\end{array}
\right.
\]
and the decreasing function%
\[
D\left(  x\right)  =\int_{x}^{\infty}d\left(  t\right)
\mathrm{d}{t}=\left\{
\begin{array}
[c]{ll}%
h\left(  a^{\prime\prime}\right)  -h\left(  \infty\right)  , &
x<a^{\prime
\prime}\\
h\left(  x\right)  -h\left(  \infty\right)  , & x\geq a^{\prime\prime}%
\end{array}
\right.  .
\]
\newline Let us rewrite condition (\ref{prduco1}) as follows:%
\begin{equation}
c_{0}+C\left(  x\right)  +D\left(  x\right)  \leq h\left(  x\right)
.\label{eq:dualconst}%
\end{equation}
If $x<a^{\prime}$, equation (\ref{eq:dualconst}) becomes 
\[
c_{0}+h\left(  x\right)  -h\left(  a^{\prime}\right)  +h\left(
a^{\prime \prime}\right)  -h\left(  \infty\right)  \leq h\left( x\right)
.
\]
And, replacing the inequality by an equality (simply taking the upper bound of the constraint), we obtain %
\begin{equation}
h\left(  a^{\prime\prime}\right)  =h\left(  a^{\prime}\right)
.\label{prduco3}%
\end{equation}
If $a^{\prime}<x<a^{\prime\prime}$, we have $c_{0}+h\left(
a^{\prime\prime }\right)  -h\left(  \infty\right)  \leq h\left(
x\right)  $ which means that for all $x\in\left(
a^{\prime},a^{\prime\prime}\right)  $ we have $h\left(
a^{\prime\prime}\right)  (=h\left(  a^{\prime}\right)  )\leq h\left(
x\right)  $ (i.e. $h\left(  a^{\prime\prime}\right)  $ and
$a^{\prime}$ are the minimal values of the function $h\left(
x\right)  $ in interval $x\in\left(
a^{\prime},a^{\prime\prime}\right)  $.) If $x\geq a^{\prime
\prime}$, then we get the trivial equality $c_{0}+h\left(  x\right)
-h\left( \infty\right)  =h\left(  x\right)  $. The two proposed
solutions are valid iff there exist solutions to Eq. (\ref{prduco2}) and Eq. (\ref{prduco3}), respectively for the primal and dual
problem. That such solutions exist can be seen by considering Lemma\ref{lem:equationsol} and taking $a^{\prime}=\overline{F}^{-1}\left(  \gamma\right)  $ and
$a^{\prime\prime}=\underline{F}^{-1}\left(  \gamma\right)  $, with $\gamma$ the solution of Eq.~\eqref{NonMonF38}. We then find the admissible
values of the objective functions%
\[
v_{\min}=\int_{0}^{a^{\prime}}h\left(  x\right)  \mathrm{d}\overline{F}%
+\int_{a^{\prime\prime}}^{\infty}h\left(  x\right)
\mathrm{d}\underline{F},
\]%
\[
w_{\max}=c_{0}+\int_{0}^{\infty}\left(  -c\left(  t\right) \overline
{F}\left(  t\right)  +d\left(  t\right) \underline{F}\left(
t\right) \right)  \mathrm{d}{t}.
\]
By using integration by parts together with equations (\ref{prduco2}%
)-(\ref{prduco3}), we can show that equality $w_{\max}=v_{\min}$
holds, with $\gamma$ the particular solution of equation
(\ref{NonMonF38}) for which optimum is reached, as was to be proved.
\end{proof}

\begin{proof}
[\textbf{Proof using random sets}]Let us now consider equations
(\ref{eq:upperprevleb})-(\ref{eq:Lowerprevleb}). Looking first at
equation (\ref{eq:upperprevleb}), we see that before
$\gamma=\underline{F}(a)$, the supremum of $h$ on $A_{\gamma}$ is
$h(a^{*}_{\gamma})$, since $h$ is increasing between $[\infty,a]$.
Between $\gamma=\underline{F}(a)$ and $\gamma=\overline{F}(a)$, the
supremum of $h$ on $A_{\gamma}$ is $f(a)$. After
$\gamma=\overline{F}(a)$, we can make the same reasoning as for the
increasing part of $h$ (except that it is now decreasing). Finally,
this gives us the
following formula:%
\begin{equation}
\overline{\mathbb{E}}(h) =  \int\limits_{0}^{\underline{F}(a)}
 h(a^{*}_{\gamma}) d\gamma+ 
\int\limits_{\underline{F}(a)}^{\overline {F}(a)}  h(a)
d\gamma+  \int\limits_{\overline{F}(a)}^{1}  h(a_{*
\gamma}) d\gamma
\end{equation}
which is equivalent to (\ref{eq:unilebupper}). Let us now turn to
the lower expectation. Before $\gamma=\underline{F}(a)$ and after
$\gamma=\overline {F}(a)$, finding the infinimum is again not a
problem (it is respectively $h(a_{* \gamma})$ and
$h(a^{*}_{\gamma})$). Between $\gamma=\underline{F}(a)$ and
$\gamma=\overline{F}(a)$, since we know that $h$ is increasing
before $x=a$ and decreasing after, infinimum is either $h(a_{*
\gamma})$ or $h(a^{*}_{\gamma})$. This gives us equation
%\begin{align}
%\underline{\mathbb{E}}h = \int\limits_{0}^{\underline{F}(a)} h(a_{*
%\gamma}) d\gamma+  &
%\int\limits_{\underline{F}(a)}^{\overline{F}(a)} \min(h(a_{*
%\gamma}),h(a^{*}_{\gamma})) d\gamma+\nonumber\\
%&  \int\limits_{\overline{F}(a)}^{1} h(a^{*}_{\gamma}) d\gamma
%\label{eq:unilowerappr}%
%\end{align}
\begin{equation}
 \underline{\mathbb{E}}h = 
\int\limits_{0}^{\underline{F}(a)} h(a_{* \gamma}) d\gamma+
\int\limits_{\underline{F}(a)}^{\overline{F}(a)} \min(h(a_{*
\gamma}),h(a^{*}_{\gamma})) d\gamma+
\int\limits_{\overline{F}(a)}^{1} h(a^{*}_{\gamma}) d\gamma
\label{eq:unilowerappr}%
\end{equation}
and if we use equations (\ref{prduco2}),(\ref{prduco3}) as in the
first proof (reasoning used in the first proof to show that they
have a solution is general, and thus applicable here), we know that
there is a level $\alpha$ s.t.
$h(\overline{F}^{-1}(\alpha))=h(\underline{F}^{-1}(\alpha))$, and
for which the above equation simplify in equation (\ref{NonMonF38}).
\end{proof}

Figure~\ref{fig:onemax-uncon} shows that the extremizing distribution corresponding to upper expectation consists in concentrating as much probability mass as possible on the maximum, as could have been expected, while the cumulative distribution reaching the lower expectation consists of an horizontal jump avoiding higher values. As we shall see, finding the level $\alpha$ satisfying Equation (\ref{prduco2}) and at which this jump occurs is sometimes feasible, and in this case exact lower and upper expectations can be found. In other cases, when computing the upper expectation by numerical methods and linear programming, results indicate that it is important to include the value $a$ corresponding to the maximum of $h$ in the sampled value, as well as values close to it when computing the upper expectation. When using the random set approach, they show that there are no need to consider values $\gamma$ inside the interval $[\ldf(a),\udf(a)]$, the bounds being sufficient. For the lower expectation, results indicate that when using linear programming, it is preferable to sample outside the interval $[\udf^{-1}(\alpha),\ldf^{-1}(\alpha)]$.

However, it can happens that the exact value of $\alpha$ cannot be computed, but that the integrals in Eq.\eqref{NonMonF35}-\eqref{NonMonF36} can still be solved. In this case, lower and upper expectations have to be approximated, for example by  scanning  a more or less wide range of possible values for $\alpha$ (see~\cite{Utkin06} for an example).

\begin{example} \label{exmp:onemax-uncon}We still consider the same p-box as in Example~\ref{exmp:mono-uncon}, but we now suppose that the loss is modelled by the function $h(x)=60-(x-5)^{2}$. This loss function can express the idea that it is preferable for the unit to fail when it begins to work or when it has worked for a long time, rather than when it works at full capacity, as the cost of slowing a whole production line would then be quite higher. $h$ has one maximum at $a=5$, and we get
\begin{align*}
\overline{\mathbb{E}}h  &  =h(5)\left[  \overline{F}(5)-\underline
{F}(5)\right]  +\int_{0}^{5}h(x)\mathrm{d}\underline{F}(x)+\int_{5}^{\infty
}h(x)\mathrm{d}\overline{F}(x)\\
&  =60\cdot\left(  \exp(-0.2\cdot5)-\exp(-0.5\cdot5)\right)  +31.321+4.268\\
&  =52.736.
\end{align*}
Since $\overline{F}^{-1}(\alpha)=-2\ln(1-\alpha)$ and $\underline{F}%
^{-1}(\alpha)=-5\ln(1-\alpha)$, then $\alpha$ can be found by solving the
following equality
\[
60-(-2\ln(1-\alpha)-5)^{2}=60-(-5\ln(1-\alpha)-5)^{2}.
\]
Hence, we have two solutions $\alpha=1-\exp(-10/7)$ and $\alpha=0$. Since
$\overline{F}^{-1}(0)=\underline{F}^{-1}(0)$, then the second solution has to
be removed. Therefore, we get $\alpha=1-\exp(-10/7)=\allowbreak0.76$. Hence,
we obtain
\begin{align*}
\underline{\mathbb{E}}h  &  =\int_{-\infty}^{-2\ln(1-0.76)}h(x)\mathrm{d}%
\overline{F}(x)+\int_{-5\ln(1-0.76)}^{\infty}h(x)\mathrm{d}\underline{F}(x)\\
&  =\int_{-\infty}^{2.85}\left(  60-(x-5)^{2}\right)  0.5e^{-0.5x}%
\mathrm{d}x+\int_{7.\,\allowbreak14}^{\infty}\left(  60-(x-5)^{2}\right)
0.2e^{-0.2x}\mathrm{d}x\\
&  =29.745.
\end{align*}
Finally, we obtain the interval of expected losses $[29.745,$ $52.736]$. Using the random set approach, we get 
\begin{align*}
\overline{\mathbb{E}}(h)    = &\int\limits_{0}^{1-\exp
(-0.5\cdot5)}\left(  60-(-5\ln(1-\gamma)-5)^{2}\right)
\mathrm{d}\gamma  +h(5)\left[  \overline{F}(5)-\underline{F}(5)\right]  \\
&  +\int\limits_{1-\exp(-0.2\cdot5)}^{1}\left(
60-(-2\ln(1-\gamma)-5)^{2}\right)  \mathrm{d}\gamma\\
 = & \ \ 52.736.
\end{align*}%
\begin{align*}
\underline{\mathbb{E}}(h)   = & \int\limits_{0}^{\allowbreak0.76}\left(
60-(-5\ln(1-\gamma)-5)^{2}\right)  \mathrm{d}\gamma+\int\limits_{0.76}%
^{1}\left(  60-(-2\ln(1-\gamma)-5)^{2}\right)  \mathrm{d}\gamma\\
   = & \ \ 29.745.
\end{align*}
\end{example}

If the function $h$ is symmetric about $a$, i.e., the equality
$h(a-x)=h(a+x) $ is valid for all $x\in\mathbb{R}$, then the value
of $\alpha$ in (\ref{NonMonF38}) does not depend on $h$ and is
determined as
\[
a-\overline{F}^{-1}(\alpha)=\underline{F}^{-1}(\alpha)-a.
\]
Note that expressions (\ref{NonMonF33}),(\ref{NonMonF34}) can be
obtained from (\ref{NonMonF35}),(\ref{NonMonF36}) by taking
$a\rightarrow\infty$.

\subsection{Conditional expectations}

We now consider conditioning by an event $B=[b_0,b_1]$, while $h$ is still assumed to have one maximum. The following proposition indicates how lower and upper conditional expectations can be computed in this case.

\begin{proposition} If the function $h$ has one maximum at point
$a\in\mathbb{R}$, then the upper and lower conditional expectations
of $h(X)$
on $[\underline{F},\overline{F}]$ after observing the event $B$ are%
\begin{align*}
\overline{\mathbb{E}}(h|B)=\sup_{\substack{\underline{F}(b_{0})\leq\alpha
\leq\overline{F}(b_{0})\\\underline{F}(b_{1})\leq\beta\leq\overline{F}(b_{1}%
)}}\frac{1}{\beta-\alpha}\Psi(\alpha,\beta),\\
\underline{\mathbb{E}}(h|B)=\inf_{\substack{\underline{F}(b_{0})\leq\alpha
\leq\overline{F}(b_{0})\\\underline{F}(b_{1})\leq\beta\leq\overline{F}(b_{1}%
)}}\frac{1}{\beta-\alpha}\Phi(\alpha,\beta),
\end{align*}
with
\begin{align*}
\Psi(\alpha,\beta)  
= & \ \ I_{(\alpha<\underline{F}^{-1}(a))}\int_{\underline
{F}^{-1}(\alpha)}^{a}h(x)\mathrm{d}\underline{F} 
+I_{(\beta   >\overline{F}^{-1}(a))}\int_{a}^{\overline{F}^{-1}(\beta
)}h(x)\mathrm{d}\overline{F}\\
& \ \  +h(a)\left(
\min(\overline{F}(a),\beta)-\max(\underline{F}(a),\alpha
)\right) \\
\Phi(\alpha,\beta)    = & \ \  h(b_{0})\left(
\overline{F}(b_{0})-\alpha\right)
+\int_{b_{0}}^{\overline{F}^{-1}(\varepsilon)}h(x)\mathrm{d}\overline{F}\\
& \ \  +h(b_{1})\left(  \beta-\underline{F}(b_{1})\right)
+\int_{\underline
{F}^{-1}(\varepsilon)}^{b_{1}}h(x)\mathrm{d}\underline{F}
\end{align*}
Here $I_{(a<b)}$ is the indicator function taking $1$ if $a<b$ and $0$
if $a\geq b$; $\varepsilon$ is one of the roots of the following
equation:
\begin{equation} \label{eq:epsi-cond}
h\left(  \overline{F}^{-1}(\varepsilon)\right)  =h\left(  \underline{F}%
^{-1}(\varepsilon)\right)  .
\end{equation}

\end{proposition}

\begin{proof}
The proof follows from Proposition~\ref{prop:condexp} where $\Psi(\alpha,\beta),\Phi(\alpha,\beta)$ are respectively replaced by formulas given in Proposition~\ref{pr:NonMonF1}.
\end{proof}

\begin{example} \label{exmp:onemax-cond} We consider the same $h$ as in Example~\ref{exmp:onemax-uncon}, the same p-box $\pbox$ as in the other examples, and the conditioning event $B=[1,8]$. From Example~\ref{exmp:onemax-uncon}, the solutions of Eq.~\eqref{eq:epsi-cond} are $\varepsilon=1-\exp(-10/7)=\allowbreak0.76$, $\underline{F}^{-1} (\varepsilon)=7.14$, $\overline{F}^{-1}(\varepsilon)=2.85$. We also have $a=5$, $\underline{F}(a)=1-\exp(-0.2\cdot5)=\allowbreak0.63$, $\overline{F}(a)=1-\exp(-0.5\cdot
5)=\allowbreak0.92\allowbreak$. Let us first concentrate on 
\[
\overline{\mathbb{E}}(h|B)=\sup_{\substack{0.18\leq\alpha\leq0.39\\0.8\leq
\beta\leq0.98}}\frac{1}{\beta-\alpha}\Psi(\alpha,\beta),
\]
where 
\begin{align*}
\Psi(\alpha,\beta) =&  \ \ I_{(\alpha<0.63)}\int_{-5\ln(1-\alpha)}^{5}\left(
60-(x-5)^{2}\right)  0.2e^{-0.2x}\mathrm{d}x\\
+ & \ \ I_{(\beta   >0.92)}\int_{5}^{-2\ln(1-\beta)}\left(  60-(x-5)^{2}\right)
0.5e^{-0.5x}\mathrm{d}x\\
+ & \ \ 60\left(  \min(1-e^{-0.5\cdot5},\beta)-\max(1-e^{-0.2\cdot5}%
,\alpha)\right)  \\
=&  \ \ \left(  \allowbreak25\alpha\ln^{2}\left(  1-\alpha\right)  -25\ln
^{2}\left(  1-\alpha\right)  -\allowbreak35\alpha+31.32\right) +60\left(  \min\left(  0.92,\beta\right)  -0.63\allowbreak\right) \\
& +I_{(\beta  >0.92)}\left(  4\allowbreak\left(  1-\beta\right)  \ln^{2}\left(
1-\beta\right)  +12\left(  1-\beta\right)  \ln\left(  1-\beta\right)
+47\beta-42.73\right)  
\end{align*}
since $0.18\leq\alpha\leq0.39$, we have $I_{(\alpha<0.63)}=1$. Let us then consider the two sets of value $[0.8,0.92]$ and $(0.92,0.98]$ for which $I_{(\beta   >0.92)}$ takes different values, and the respective functions $\Psi_{1}(\alpha,\beta)$,$\Psi_{2}(\alpha,\beta)$ associated to them: 
\begin{align*}
\Psi_{1}(\alpha,\beta) &  =\allowbreak25\alpha\ln^{2}\left(  1-\alpha\right)
-25\ln^{2}\left(  1-\alpha\right)  -\allowbreak35\alpha+31.32+60\left(
\beta-0.63\allowbreak\right)  
%&  =60\beta-35\alpha-25\ln^{2}\left(  1-\alpha\right)  +\allowbreak25\alpha
%\ln^{2}\left(  1-\alpha\right)  -6.48,
\end{align*}
\begin{align*}
\Psi_{2}(\alpha,\beta) &  =25\alpha\ln^{2}\left(  1-\alpha\right)  -25\ln
^{2}\left(  1-\alpha\right)  -\allowbreak35\alpha+31.32\\
&  +4\allowbreak\left(  1-\beta\right)  \ln^{2}\left(  1-\beta\right)
+12\left(  1-\beta\right)  \ln\left(  1-\beta\right)  +47\beta
-42.73+\allowbreak17.4
%&  =47\beta-35\alpha+12\left(  1-\beta\right)  \ln\left(  1-\beta\right)
%-\allowbreak25(1-\alpha)\ln^{2}\left(  1-\alpha\right)  +4\left(
%1-\beta\right)  \ln^{2}\left(  1-\beta\right)  +5.99.
\end{align*}
It can be checked that the derivative $\nicefrac{\mathrm{d}\Psi_{1}(\alpha,\beta)/(\beta-\alpha)}{\mathrm{d}\beta
}$ is positive for $0.18\leq\alpha\leq0.39$, hence the maximum of $\Psi_{1}(\alpha,\beta
)/(\beta-\alpha)$ is achieved at $\beta=0.98$. Also, since $\Psi_{1}%
(\alpha,0.98)/(0.98-\alpha)$ decreases as $\alpha$ increases, we have
\[
\sup\frac{1}{\beta-\alpha}\Psi_{1}(\alpha,\beta)=\frac{1}{0.98-0.18}\Psi
_{1}(0.18,0.98)=56.52.
\]
A similar analysis for $\nicefrac{\Psi_{2}(\alpha,\beta)}{(\beta
-\alpha)}$ shows that maximum is achieved for $\alpha=0.39$, $\beta=0.8$. Hence%
\[
\sup\frac{1}{\beta-\alpha}\Psi_{2}(\alpha,\beta)=\frac{1}{0.8-0.39}\Psi
_{2}(0.39,0.8)=59.57.
\]
and, finally, we have $\overline{\mathbb{E}}(h|B)=\max(56.52,59.57)=59.57.$  Figure~\ref{fig:NonMonF1} gives an illustration of the extremizing cumulative distribution for which this upper conditional expectation is reached.

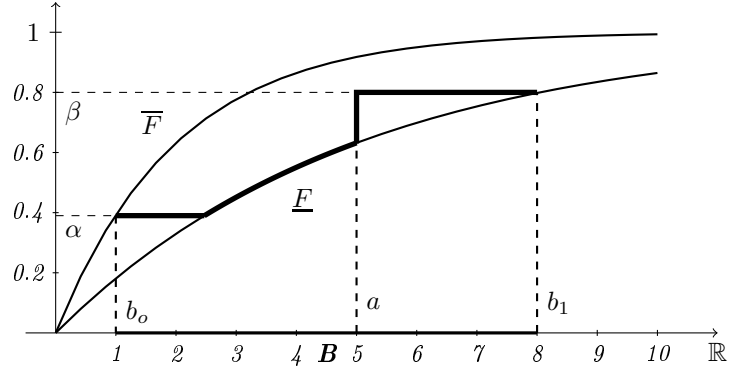
\begin{figure}[ptb]
\centering

\begin{tikzpicture}[scale=0.8]

\draw[->] (-0.5,0) -- (11,0) node[below] {$\mathbb{R}$};
\draw[->] (0,0) -- (0,5.5);
\draw (-0.1,5) node[left] {$1$} -- (0.1,5);
\draw[domain=0:10,thick]  plot[id=exp7] function{5*(1-exp(-0.2*x))} ;
\node at (1.6,3.5) {$\udf$};
%\draw[domain=-1:5,ultra thick,dashed]  plot[id=norm] function{3*norm((x-1)*2)};
\draw[domain=0:10,thick]  plot[id=exp8] function{5*(1-exp(-0.5*x))};
\draw[dashed] (0,0.39*5) node[below right] {$\alpha$} -- (1,0.39*5);
\draw[dashed] (0,0.8*5) node[below right] {$\beta$} -- (8,0.8*5);
\draw[dashed,thick] (1,0.39*5) -- (1,0) node[above right] {$b_o$};
\draw[very thick]  (1,0) -- (8,0) node[midway,below] {\textbf{B}};
\draw[dashed,thick] (8,0)  -- (8,0.8*5) node[very near start,right] {$b_1$};
\draw[dashed,thick] (5,0)  -- (5,0.8*5) node[very near start,right] {$a$};
\foreach \x in {1,...,10} {\draw (\x , 0.05) -- (\x , -0.05) node[below] {\small \x};}
\foreach \y in {0.2,0.4,0.6,0.8} {\draw (0.05 , \y*5) -- (-0.05 , \y*5) node[left] {\small \y};}
\draw[line width=2pt] (1,0.39*5) -- (2.47,0.39*5);
\draw[domain=2.47:5, line width=2pt]  plot[id=exp9] function{5*(1-exp(-0.2*x))} ;
\draw[line width=2pt] (5,0.63*5) -- (5,0.8*5) -- (8,0.8*5);
%\draw[thick,dotted] (-0.1,1.5) node[left] {$\gamma$} -- (4,1.5) ; 
%\draw[thick,dotted] (1.5,1.5) -- (1.5,-0.2) node[below] {$a_{\ast\gamma}$} ;
%\draw[thick,dotted] (3,1.5) -- (3,-0.2) node[below] {$a^{\ast}_\gamma$};
%\draw[very thick] (1.5,1.5) -- (3,1.5) node[midway,above] { \Large $A_{\gamma}$};
\node at (4.1,2.2) {$\ldf$};
%\node at (5,-2) {Optimal $\df$ for $\overline{\mathbb{E}}(h|B)$ };

\end{tikzpicture}
\caption{Optimal
distribution (thick) for computing upper conditional expectation on $B=[1,8]$}%
\label{fig:NonMonF1}%
\end{figure}

Let us now detail the computations for 
\[
\underline{\mathbb{E}}(h|B)=\inf_{\substack{0.18\leq\alpha\leq0.39\\0.8\leq
\beta\leq0.98}}\frac{1}{\beta-\alpha}\Phi(\alpha,\beta),
\]
where%
\begin{align*}
\Phi(\alpha,\beta)  &  =\left(  60-(1-5)^{2}\right)  \left(  0.39-\alpha
\right)  +\int_{1}^{2.85}\left(  60-(x-5)^{2}\right)  0.5e^{-0.5x}%
\mathrm{d}x\\
&  +\left(  60-(8-5)^{2}\right)  \left(  \beta-\allowbreak0.8\right)
+\int_{7.14}^{8}\left(  60-(x-5)^{2}\right)  0.2e^{-0.2x}\mathrm{d}x\\
&  =\allowbreak51\beta-44\alpha-3.54.
\end{align*}
The function $\frac{1}{\beta-\alpha}\Phi(\alpha,\beta)$ increases as $\alpha$
increases by arbitrary $0.8\leq\beta\leq0.98$ and increases as $\beta$
increases. This implies that $
\underline{\mathbb{E}}(h|B)=\nicefrac{1}{(0.8-0.18)}\left(  
51\cdot0.8-44\cdot0.18-3.54\right)  =47.32.$
\end{example}

Note that, in the general case, four functions $\Psi_{i}$ (corresponding to all combinations of values of $I_{(\alpha<\underline{F}^{-1}(a))}$, $I_{(\beta   >\overline{F}^{-1}(a))}$ inside $\{0,1\}^2$) would have to be considered in the computation of $\overline{\mathbb{E}}(h|B)$. Example~\ref{exmp:onemax-cond} well illustrates the fact that when $h$ is non-monotone, analytical solutions can still be found in some cases, but that they tend to become tedious to compute. This will be confirmed in the next section.

\section{Functions with local maxima/minima
\label{sec:maxminunivar}}

Now we consider a general form of the function $h$, i.e., the
function $h\left(  x\right)  $ has alternate local maxima at point
$a_{i}$, $i=1,2,...$ and minima at point $b_{i}$, $i=0,1,2,...$, such that
\begin{equation}
b_{0}<a_{1}<b_{1} \ldots < b_{i} < a_{i} < b_{i+1} < \ldots \label{NonMonF52}%
\end{equation}
Note that, in this case, studying the shape of the extremizing cumulative distribution reaching lower expectation is sufficient, thanks to the duality between lower and upper expectation.

\begin{proposition}
\label{prop:manymaxmin} If local maxima ($a_{i}$) and minima
($b_{i}$) of the function $h$ satisfy condition (\ref{NonMonF52}),
then the extremizing distribution $F$ for computing the lower
unconditional expectation $\underline{\mathbb{E}}(h)$ has discontinuities (vertical jumps) at
points $b_{i}$, $i=1,...$. of the size
\[
\min\left(  \overline{F}\left(  b_{i}\right)  ,\alpha_{i+1}\right)
-\max\left(  \underline{F}\left(  b_{i}\right)  ,\alpha_{i}\right) .
\]

Between points $b_{i-1}$ and $b_{i}$, that is between discontinuities numbered $i-1$ and $i$, the extremizing cumulative probability
distribution function $F$ is of the form:
\[
F\left(  x\right)  =\left\{
\begin{array}
[c]{ll}%
\overline{F}\left(  x\right)  , & x<a^{\prime}\\
\alpha, & a^{\prime}\leq x\leq a^{\prime\prime}\\
\underline{F}\left(  x\right)  , & a^{\prime\prime}<x
\end{array}
\right.  ,
\]
where $\alpha$ is the root of the equation
\[
h\left(  \max\left(  \overline{F}^{-1}\left(  \alpha\right)
,b_{i-1}\right) \right)  =h\left(  \min\left(
\underline{F}^{-1}\left(  \alpha\right) ,b_{i}\right)  \right)
\]
in interval $\left[  \underline{F}\left(  a_{i}\right)
,\overline{F}\left( a_{i}\right)  \right]  $, and $a^{\prime}$,$a^{\prime\prime}$ are such that
\[
a^{\prime}=\max\left(  \overline{F}^{-1}\left(  \alpha\right)  ,b_{i-1}%
\right)  ,\ a^{\prime\prime}=\min\left(  \underline{F}^{-1}\left(
\alpha\right)  ,b_{i}\right)  .
\]

The upper expectation $\overline{\mathbb{E}}(h)$ can be found from
the condition
$\overline{\mathbb{E}}(h)=-\underline{\mathbb{E}}(-h)$.
\end{proposition}

\begin{proof}
[\textbf{Proof using linear programming}]This
proof is based on the investigation of the following local primal
and dual optimization problems for computing the lower expectation
of $h$ in finite interval $\left[
{b_{0},b_{1}}\right)  $ where $h$ has one maximum at point $a_{1}$:%

\begin{center}
\begin{tabular}
[c]{l}\hline $\text{\textbf{Primal problem:}}$\\\hline
\textbf{Min.} $\mathbf{v}=\int_{b_{0}}^{b_{1}}{h\left(  x\right)  f\left(  x\right)  }\mathrm{d}%
{x}$ \\ $\text{subject to}$\\
$f\left(  x\right)  \geq0,$ $F_{0}\geq0,$ $F_{1}\geq0,$\\
$-\int_{b_{0}}^{x}{f\left(  t\right)
\mathrm{d}{t}}-F_{0}\geq-\overline
{F}\left(  x\right)  ,$\\
$\int_{b_{0}}^{x}{f\left(  t\right)
\mathrm{d}{t}}+F_{0}\geq\underline
{F}\left(  x\right)  ,$\\
$-F_{0}\geq-\overline{F}\left(  {b_{0}}\right)
,$$F_{0}\geq\underline
{F}\left(  {b_{0}}\right)  ,$\\
$-F_{1}\geq-\overline{F}\left(  {b_{1}}\right)
,$$F_{1}\geq\underline
{F}\left(  {b_{1}}\right)  ,$\\
$\int_{b_{0}}^{b_{1}}{f\left(  t\right)  \mathrm{d}{t}}+F_{0}-F_{1}%
=0.$\\\hline
\end{tabular}
\begin{tabular}
[c]{l}\hline $\text{\textbf{Dual problem:}}$\\\hline
\textbf{Max.} $\mathbf{w}=-c_{0}\overline{F}\left(  {b_{0}}\right)
+d_{0}\underline{F}\left( {b_{0}}\right)  -c_{1}\overline{F}\left(
{b_{1}}\right) $  \\
$+d_{1}\underline
{F}\left(  {b_{1}}\right) +\int_{b_{0}}^{b_{1}}\left(  {-\overline{F}\left(  x\right) c\left(
x\right)  +\underline{F}\left(  x\right)  d\left(  x\right) }\right)
{\mathrm{d}x}$\\
$\text{subject to}$\\
$e+\int_{x}^{b_{1}}{\left(  {-c\left(  t\right)  +d\left(  t\right)
}\right)
\mathrm{d}{t}\leq}h\left(  x\right)  ,$\\
$e-c_{0}+d_{0}+\int_{b_{0}}^{b_{1}}{\left(  {-c\left(  t\right)
+d\left(
t\right)  }\right)  \mathrm{d}{t}\leq}0,$\\
$-e-c_{1}+d_{1}\leq0,$\\
$c\left(  x\right)  \geq0,$$c_{0}\geq0,$$c_{1}\geq0,$\\
$d\left(  x\right)
\geq0,$$d_{0}\geq0,$$d_{1}\geq0,$$e\in\mathbb{R}$\\\hline
\end{tabular}
\end{center}

The optimal solutions of the above problems correspond to the extremizing distribution for values $x \in\left[  {b_{0},b_{1}}\right)$. $F_0:=F(b_0)$ and $F_1:=F(b_1)$ respectively stand for the values of the extremizing $F$ in $b_0$ and $b_1$. The proof then follows in two main steps:
\begin{enumerate}
\item Find optimal solution (that is, propose a feasible solution which coincide for both the primal and dual problem) for the above primal and dual problems, and consequently the values of the extremizing $F$ between any two local minima $[b_i,b_{i+1}]$
\item Show that the combination of these piece-wise extremizing $F$ correspond to a cumulative distribution.
\end{enumerate}

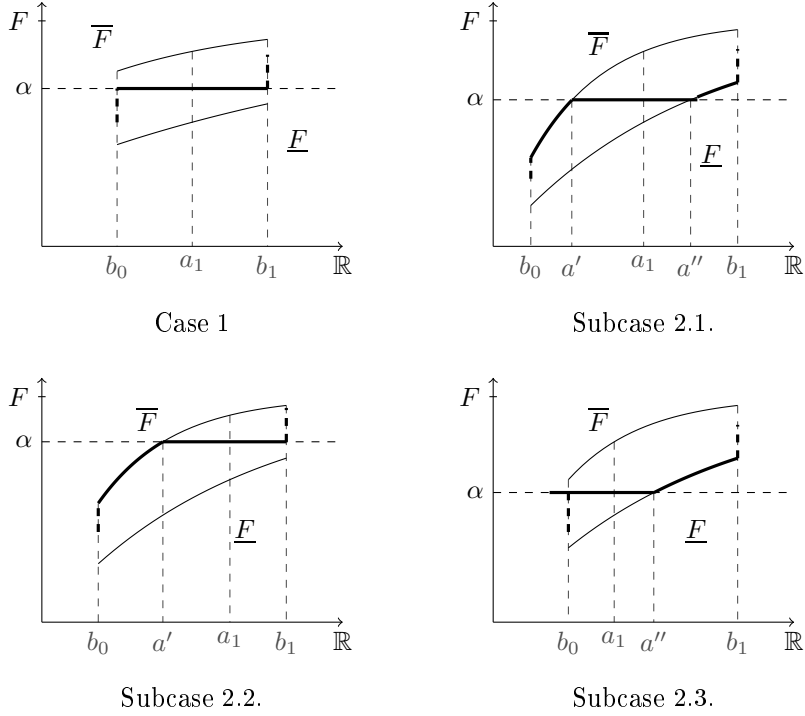
\begin{figure}
\begin{tikzpicture}
\draw[->] (2,0) -- (6,0) node[below] {$\mathbb{R}$};
\draw[->] (2,0) -- (2,3.25);
\draw (1.95,3) node[left] {$F$} -- (2.05,3);
\draw[domain=3:5]  plot[id=exp10] function{3*(1-exp(-0.2*x))} ;
\node at (2.8,2.8) {$\udf$};
%\draw[domain=-1:5,ultra thick,dashed]  plot[id=norm] function{3*norm((x-1)*2)};
\draw[domain=3:5]  plot[id=exp11] function{3*(1-exp(-0.5*x))};
\draw[dashed] (2,0.7*3) node[left] {$\alpha$} -- (6,0.7*3);
\draw[dashed,opacity=0.7] (3,0.776*3) -- (3,0) node[below] {$b_0$};
\draw[dashed,opacity=0.7] (4,0.864*3) -- (4,0) node[below] {$a_1$};
\draw[dashed,opacity=0.7] (5,0.918*3) -- (5,0) node[below] {$b_1$};
\draw[dashed,very thick] (3,0.55*3) -- (3,0.7*3);
\draw[very thick] (3,0.7*3) -- (5,0.7*3);
\draw[dashed,very thick] (5,0.7*3) -- (5,0.85*3);
\node at (5.4,1.4) {$\ldf$};
\node at (4,-1) {Case 1};

\begin{scope}[yshift=-5cm,xshift=2cm,scale=0.5]
\draw[->] (0,0) -- (8,0) node[below] {$\mathbb{R}$};
\draw[->] (0,0) -- (0,6.5);
\draw (-0.1,6) node[left] {$F$} -- (0.1,6);
\draw[domain=1.5:6.5]  plot[id=exp12] function{6*(1-exp(-0.2*x))} ;
\node at (2.8,5.4) {$\udf$};
%\draw[domain=-1:5,ultra thick,dashed]  plot[id=norm] function{3*norm((x-1)*2)};
\draw[domain=1.5:6.5]  plot[id=exp13] function{6*(1-exp(-0.5*x))};
\draw[dashed] (0,0.8*6) node[left] {$\alpha$} -- (8,0.8*6);
\draw[dashed,opacity=0.7] (1.5,0.527*6) -- (1.5,0) node[below] {$b_0$};
\draw[dashed,opacity=0.7] (5,0.918*6) -- (5,0) node[below] {$a_1$};
\draw[dashed,opacity=0.7] (3.218,0.8*6) -- (3.218,0) node[below] {$a'$};
\draw[dashed,opacity=0.7] (6.5,0.961*6) -- (6.5,0) node[below] {$b_1$};
\draw[dashed,very thick] (1.5,0.4*6) -- (1.5,0.527*6);
\draw[very thick,domain=1.5:3.218]  plot[id=exp14] function{6*(1-exp(-0.5*x))} ;
\draw[very thick] (3.218,0.8*6) -- (6.5,0.8*6);
\draw[dashed,very thick] (6.5,0.8*6) -- (6.5,0.95*6);
\node at (5.4,2.4) {$\ldf$};
\node at (4,-2) {Subcase 2.2.};

\end{scope}

\begin{scope}[yshift=-5cm,xshift=8cm,scale=0.5]
\draw[->] (0,0) -- (8,0) node[below] {$\mathbb{R}$};
\draw[->] (0,0) -- (0,6.5);
\draw (-0.1,6) node[left] {$F$} -- (0.1,6);
\draw[domain=2:6.5]  plot[id=exp15] function{6*(1-exp(-0.2*x))} ;
\node at (2.8,5.4) {$\udf$};
%\draw[domain=-1:5,ultra thick,dashed]  plot[id=norm] function{3*norm((x-1)*2)};
\draw[domain=2:6.5]  plot[id=exp16] function{6*(1-exp(-0.5*x))};
\draw[dashed] (0,0.575*6) node[left] {$\alpha$} -- (8,0.575*6);
\draw[dashed,opacity=0.7] (2,0.632*6) -- (2,0) node[below] {$b_0$};
\draw[dashed,opacity=0.7] (4.27,0.575*6) -- (4.27,0) node[below] {$a''$};
\draw[dashed,opacity=0.7] (3.218,0.8*6) -- (3.218,0) node[below] {$a_1$};
\draw[dashed,opacity=0.7] (6.5,0.961*6) -- (6.5,0) node[below] {$b_1$};
\draw[dashed,very thick] (2,0.4*6) -- (2,0.575*6);
\draw[very thick,domain=4.27:6.5]  plot[id=exp17] function{6*(1-exp(-0.2*x))} ;
\draw[very thick] (1.5,0.575*6) -- (4.27,0.575*6);
\draw[dashed,very thick] (6.5,0.727*6) -- (6.5,0.875*6);
\node at (5.4,2.4) {$\ldf$};
\node at (4,-2) {Subcase 2.3.};

\end{scope}

\begin{scope}[xshift=8cm,scale=0.5]
\draw[->] (0,0) -- (8,0) node[below] {$\mathbb{R}$};
\draw[->] (0,0) -- (0,6.5);
\draw (-0.1,6) node[left] {$F$} -- (0.1,6);
\draw[domain=1:6.5]  plot[id=exp18] function{6*(1-exp(-0.2*x))} ;
\node at (2.8,5.4) {$\udf$};
%\draw[domain=-1:5,ultra thick,dashed]  plot[id=norm] function{3*norm((x-1)*2)};
\draw[domain=1:6.5]  plot[id=exp19] function{6*(1-exp(-0.5*x))};
\draw[dashed] (0,0.65*6) node[left] {$\alpha$} -- (8,0.65*6);
\draw[dashed,opacity=0.7] (1,0.393*6) -- (1,0) node[below] {$b_0$};
\draw[dashed,opacity=0.7] (5.249,0.65*6) -- (5.249,0) node[below] {$a''$};
\draw[dashed,opacity=0.7] (2.09,0.65*6) -- (2.09,0) node[below] {$a'$};
\draw[dashed,opacity=0.7] (4,0.864*6) -- (4,0) node[below] {$a_1$};
\draw[dashed,opacity=0.7] (6.5,0.961*6) -- (6.5,0) node[below] {$b_1$};
\draw[dashed,very thick] (1,0.3*6) -- (1,0.393*6);
\draw[very thick,domain=5.429:6.5]  plot[id=exp20] function{6*(1-exp(-0.2*x))} ;
\draw[domain=1:2.09,very thick]  plot[id=exp21] function{6*(1-exp(-0.5*x))};
\draw[very thick] (2.09,0.65*6) -- (5.429,0.65*6);
\draw[dashed,very thick] (6.5,0.727*6) -- (6.5,0.875*6);
\node at (5.8,2.4) {$\ldf$};
\node at (4,-2) {Subcase 2.1.};

\end{scope}

\end{tikzpicture}
\caption{Four cases of piece-wise extremizing $F$}
\label{fig:gencase}
\end{figure}

%\begin{figure}[ptb]%
%\begin{tabular}
%[c]{cc}%
%\includegraphics[scale=0.75
%]{NMF2bis} & \includegraphics[scale=0.75]{NMF3bis}\\
%Case 1. & Subcase 2.1.\\
%\includegraphics[scale=0.75]{NMF4bis} & \includegraphics[scale=0.75]{NMF5bis}\\
%Subcase 2.2. & Subcase 2.3.\\
%&
%\end{tabular}
%\caption{Four cases of piecewise optimal $F$}%
%\label{fig:gencase}%
%\end{figure}

\textbf{Step (1) of the proof} To find optimal solution between $x \in\left[  {b_{0},b_{1}}\right] $, we will consider every possible cases. First, we can differentiate between two main cases, depending on the inequality relation between $\overline{F}\left(  {b_{0}}\right) $ and $\underline{F}\left(
{b_{1}}\right)  $.

\textbf{Case 1.} $\overline{F}\left(  {b_{0}}\right)  >\underline{F}\left(
{b_{1}}\right)  $. The optimal solution in this case is of
the form: it corresponds to the solution $f\left(  x\right)  =0$, $F\left(  x\right)  =F_{0}=F_{1}=\alpha$,
where $\alpha$ is an arbitrary number satisfying the condition $\underline
{F}\left(  {b_{1}}\right)  <\alpha<\overline{F}\left(  {b_{0}}\right)  $ for
the primal problem and to the solution $c\left(  x\right)  =d\left(  x\right)
=0$, $c_{0}=d_{0}=c_{1}=d_{1}=e=0$ for the dual problem. See Fig.~\ref{fig:gencase} for an illustration%

\textbf{Case 2. }$\overline{F}\left(  {b_{0}}\right)  \leq\underline{F}\left(
{b_{1}}\right)  $. This case is similar to the one considered in Section~\ref{sec:maxunivar}, since between $\left[  b_{0},b_{1}\right)$, $h$ has a maximum for $x={a_{1}}$ and is increasing (resp. decreasing) in $[b_0,a_1]$ (resp. $[a_1,b_1)$). We will therefore proceed in the same way as in the proof of Proposition~\ref{pr:NonMonF1} to find the optimal solution. First recall (Lemma~\ref{lem:equationsol}) that there is a value $\alpha$ which is a
root of the function
\[
\varphi\left(  \alpha\right)  =h\left(  {\max\left(  {\overline{F}^{-1}\left(
\alpha\right)  ,b_{0}}\right)  }\right)  -h\left(  {\min\left(  {\underline
{F}^{-1}\left(  \alpha\right)  ,b_{1}}\right)  }\right)
\]
with $\alpha \in \left[  {\underline{F}\left(  {a_{1}}\right)  ,\overline
{F}\left(  {a_{1}}\right)  }\right]  $. Three subcases can now occur, depending whether $\alpha$ is inside $[\overline{F}\left(  {b_{0}}\right),\underline{F}\left(
{b_{1}}\right)]$ or is higher/lower than any value in this interval. We now give details about each of these subcases, the reasoning being similar to the one in the proof of Proposition~\ref{pr:NonMonF1}. All subcases and associated extremizing distribution are illustrated in Fig.~\ref{fig:gencase}

\textbf{Subcase 2.1.} $\overline{F}\left(  {b_{0}}\right)  \leq\alpha
\leq\underline{F}\left(  {b_{1}}\right)  $ ($\alpha \in [\overline{F}\left(  {b_{0}}\right),\underline{F}\left(
{b_{1}}\right)]$). Let us denote $a^{\prime
}=\overline{F}^{-1}\left(  \alpha\right)  $, $a^{\prime\prime}=\underline
{F}^{-1}\left(  \alpha\right)  $. Then the optimal solution is of the form:
\[
f\left(  x\right)  =\left\{
\begin{array}
[c]{cc}%
{\mathrm{d}\overline{F}(x)/\mathrm{d}x,} & {b_{0}<x<a^{\prime}}\\
{0,} & {a^{\prime}\leqslant x\leqslant a^{\prime\prime}}\\
{\mathrm{d}\underline{F}\left(  x\right)  /\mathrm{d}}x{,} & {a^{\prime\prime
}<x<b_{1}}%
\end{array}
\right.  ,
\]%
\[
F_{0}=\overline{F}\left(  {b_{0}}\right)  ,\ F_{1}=\underline{F}\left(
{b_{1}}\right)  .
\]
This implies that%
\[
F\left(  x\right)  =\int_{b_{0}}^{x}{f\left(  t\right)  \mathrm{d}{t}}%
+F_{0}=\left\{
\begin{array}
[c]{cc}%
{\overline{F}\left(  x\right)  ,} & {b_{0}<x<a^{\prime}}\\
\alpha{,} & {a^{\prime}\leqslant x\leqslant a^{\prime\prime}}\\
{\underline{F}\left(  x\right)  ,} & {a^{\prime\prime}<x<b_{1}}%
\end{array}
\right.  .
\]
Let us now give the corresponding solution to the dual problem, and show that they are equal. According to relations between primal/dual problem, we have that if $a^{\prime
}<x<b_{1}$, then $c\left(  x\right)  =0$, and if $b_{0}<x<a^{\prime\prime}$,
then $d\left(  x\right)  =0$. It is obvious that $d_{0}=c_{1}=0$. Consider the constraint
\[
e+\int_{x}^{b_{1}}{\left(  {-c\left(  t\right)  +d\left(  t\right)  }\right)
\mathrm{d}{t}}\leq h\left(  x\right)
\]
for different intervals of $x$.

Let $a^{\prime\prime}<x<b_{1}$. Then there holds%
\[
e+\int_{x}^{b_{1}}{d\left(  t\right)  \mathrm{d}{t}}=h\left(  x\right)  .
\]
Hence $d\left(  x\right)  =-h^{\prime}\left(  x\right)  $ and $e=h\left(
{b_{1}}\right)  $.

Let $a^{\prime}\leq x\leq a^{\prime\prime}$. Then the following inequality
\[
e+\int_{a^{\prime\prime}}^{b_{1}}{d\left(  t\right)  \mathrm{d}{t}}\leq
h\left(  x\right)
\]
or $h\left(  {a^{\prime\prime}}\right)  \leq h\left(  x\right)  $ has to be
valid. Indeed, the inequality is valid due to the condition $h\left(
{a^{\prime}}\right)  =h\left(  {a^{\prime\prime}}\right)  $.

Let $b_{0}<x<a^{\prime}$. Then
\[
e-\int_{x}^{a^{\prime}}{c\left(  t\right)  \mathrm{d}{t}}+\int_{a^{\prime
\prime}}^{b_{1}}{d\left(  t\right)  \mathrm{d}{t}}=h\left(  x\right)
\]
or
\[
-\int_{x}^{a^{\prime}}{c\left(  t\right)  \mathrm{d}{t}}+h\left(
{a^{\prime\prime}}\right)  =h\left(  x\right)  .
\]
Hence $c\left(  x\right)  =h^{\prime}\left(  x\right)  $. The equality
\[
e-c_{0}+d_{0}+\int_{b_{0}}^{b_{1}}{\left(  {-c\left(  t\right)  +d\left(
t\right)  }\right)  \mathrm{d}{t}}=0
\]
shows that
\[
h\left(  {b_{1}}\right)  -c_{0}-h\left(  {a^{\prime}}\right)  +h\left(
{b_{0}}\right)  -h\left(  {b_{1}}\right)  +h\left(  {a^{\prime\prime}}\right)
=0
\]
and $c_{0}=h\left(  {b_{0}}\right)  $. It follows from the equality
$-e-c_{1}+d_{1}=0$ that there holds $d_{1}=e=h\left(  {b_{1}}\right)  $. In
sum, we have%
\[
c\left(  x\right)  =\left\{
\begin{array}
[c]{cc}%
{h^{\prime}\left(  x\right)  ,} & {b_{0}<x<a^{\prime}}\\
{0,} & {a^{\prime}\leqslant x\leqslant b_{1}}%
\end{array}
{\,}\right.  ,
\]%
\[
d\left(  x\right)  =\left\{
\begin{array}
[c]{cc}%
{0,} & {b_{0}<x<a^{\prime\prime}}\\
{-h^{\prime}\left(  x\right)  ,} & {a^{\prime\prime}\leqslant x\leqslant
b_{1}}%
\end{array}
\right.  ,
\]%
\[
c_{0}=h\left(  {b_{0}}\right)  ,\ d_{0}=c_{1}=0,\ d_{1}=e=h\left(  {b_{1}%
}\right)  .
\]

Let us now show that the two obtained solution coincide:
\[
z_{\min}=\int_{b_{0}}^{a^{\prime}}{h\left(  x\right)  \mathrm{d}\overline
{F}\left(  x\right)  }+\int_{a^{\prime\prime}}^{b_{1}}{h\left(  x\right)
\mathrm{d}\underline{F}\left(  x\right)  }%
\]%
\[
w_{\max}=-\overline{F}\left(  {b_{0}}\right)  h\left(  {b_{0}}\right)
+\underline{F}\left(  {b_{1}}\right)  h\left(  {b_{1}}\right)  -\int_{b_{0}%
}^{a^{\prime}}{\overline{F}\left(  x\right)  h^{\prime}\left(  x\right)
}\mathrm{d}{x}-\int_{a^{\prime\prime}}^{b_{1}}{\underline{F}\left(  x\right)
h^{\prime}\left(  x\right)  }\mathrm{d}{x}%
\]
or%
\begin{align*}
w_{\max}  &  =-\overline{F}\left(  {b_{0}}\right)  h\left(  {b_{0}}\right)
+\underline{F}\left(  {b_{1}}\right)  h\left(  {b_{1}}\right) \\
&  +\int_{b_{0}}^{a^{\prime}}{h\left(  x\right)  }\mathrm{d}{\overline
{F}\left(  x\right)  }-\overline{F}\left(  {a^{\prime}}\right)  h\left(
{a^{\prime}}\right)  +\overline{F}\left(  {b_{0}}\right)  h\left(  {b_{0}%
}\right) \\
&  +\int_{a^{\prime\prime}}^{b_{1}}{h\left(  x\right)  \mathrm{d}\underline
{F}\left(  x\right)  -\underline{F}\left(  {b_{1}}\right)  h\left(  {b_{1}%
}\right)  }+\underline{F}\left(  {a^{\prime\prime}}\right)  h\left(
{a^{\prime\prime}}\right) \\
&  =z_{\min}.
\end{align*}
Hence the proposed solution is the optimal one.

\textbf{Subcase 2.2.} $\alpha>\underline{F}\left(  {b_{1}}\right)  $ ($[\overline{F}\left(  {b_{0}}\right),\underline{F}\left(
{b_{1}}\right)] \leq \alpha$). Denote
$a^{\prime}=\overline{F}^{-1}\left( \alpha \right)  $. Then the optimal solution to the initial problem is:
\[
f\left(  x\right)  =\left\{
\begin{array}
[c]{cc}%
{\mathrm{d}\overline{F}\left(  x\right)  /\mathrm{d}x,} & {b_{0}<x<a^{\prime}%
}\\
{0,} & {a^{\prime}\leqslant x\leqslant b_{1}}%
\end{array}
\right.  {\,,~}F_{0}=\overline{F}\left(  {b_{0}}\right)  ,~F_{1}=\alpha  ,
\]%
\[
F\left(  x\right)  =\int_{b_{0}}^{x}{f\left(  t\right)  \mathrm{d}{t}}%
+F_{0}=\left\{
\begin{array}
[c]{cc}%
{\overline{F}\left(  x\right)  ,} & {b_{0}<x<a^{\prime}\,}\\
{\alpha  ,} & {a^{\prime}\leqslant x\leqslant
b_{1}}%
\end{array}
\right.  .
\]
The corresponding solution for the dual problem is such that if $a^{\prime
}<x<b_{1}$, then $c\left(  x\right)  =0$, and if $b_{0}<x<b_{1}$, then
$d\left(  x\right)  =0$, hence we have $d_{0}=c_{1}=0$. Again, consider 
the constraint
\[
e+\int_{x}^{b_{1}}{\left(  {-c\left(  t\right)  +d\left(  t\right)  }\right)
\mathrm{d}{t}}\leq h\left(  x\right)
\]
for different intervals. Let $a^{\prime}<x<b_{1}$. Then the condition $e\leq
h\left(  x\right)  $ must be valid. Let $b_{0}<x<a^{\prime}$. Then there holds%
\[
e-\int_{x}^{a^{\prime}}{c\left(  t\right)  \mathrm{d}{t}}=h\left(  x\right)
.
\]
Consequently, there hold the equalities $c\left(  x\right)  =h^{\prime}\left(
x\right)  $ and $e=h\left(  {a^{\prime}}\right)  $. Hence the inequality
$e=h\left(  {a^{\prime}}\right)  \leq h\left(  x\right)  $ is valid for the
interval $a^{\prime}<x<b_{1}$. The equality
\[
e-c_{0}+d_{0}+\int_{b_{0}}^{b_{1}}{\left(  {-c\left(  t\right)  +d\left(
t\right)  }\right)  \mathrm{d}{t}}=0
\]
shows that $h\left(  {a^{\prime}}\right)  -c_{0}-h\left(  {a^{\prime}}\right)
+h\left(  {b_{0}}\right)  =0$, and, therefore, $c_{0}=h\left(  {b_{0}}\right)
$. It follows from the equality $-e-c_{1}+d_{1}=0$ that there holds
$d_{1}=e=h\left(  {a^{\prime}}\right)  $. In sum, we get
\[
c\left(  x\right)  =\left\{
\begin{array}
[c]{cc}%
{h^{\prime}\left(  x\right)  ,} & {b_{0}<x<a^{\prime}}\\
{0,} & {\,a^{\prime}\leqslant x\leqslant b_{1}}%
\end{array}
\right.  ,
\]%
\[
d\left(  x\right)  =0,c_{0}=h\left(  {b_{0}}\right)  ,\ d_{0}=c_{1}%
=0,\ d_{1}=e=h\left(  {a^{\prime}}\right)  .
\]
The obtained solutions
for the primal and dual problems are such that:
\[
z_{\min}=\int_{b_{0}}^{a^{\prime}}{h\left(  x\right)  \mathrm{d}\overline
{F}\left(  x\right)  ,}%
\]%
\[
w_{\max}=-\overline{F}\left(  {b_{0}}\right)  h\left(  {b_{0}}\right)
+\udf(a') h\left(  {a^{\prime}}\right)
-\int_{b_{0}}^{a^{\prime}}{\overline{F}\left(  x\right)  h^{\prime}\left(
x\right)  }\mathrm{d}{x}%
\]
or
\begin{align*}
w_{\max}  &  =-\overline{F}\left(  {b_{0}}\right)  h\left(  {b_{0}}\right)
+\udf(a') h\left(  {a^{\prime}}\right) \\
&  +\int_{b_{0}}^{a^{\prime}}{h\left(  x\right)  }\mathrm{d}{\overline
{F}\left(  x\right)  }-\overline{F}\left(  {a^{\prime}}\right)  h\left(
{a^{\prime}}\right)  +\overline{F}\left(  {b_{0}}\right)  h\left(  {b_{0}%
}\right) \\
&  =z_{\min}.
\end{align*}
Consequently, this is the optimal solution.

\textbf{Subcase 2.3. }$\alpha<\overline{F}\left(  {b_{0}}\right)  $ ($ \alpha \leq [\overline{F}\left(  {b_{0}}\right),\underline{F}\left(
{b_{1}}\right)]$). Denote
$a^{\prime\prime}=\underline{F}^{-1}\left(  {\overline{F}\left(  {b_{0}%
}\right)  }\right)  $. Then the optimal solution to the primal problem is
\[
f\left(  x\right)  =\left\{
\begin{array}
[c]{cc}%
{0,} & {b_{0}\leqslant x\leqslant a^{\prime\prime}}\\
{\mathrm{d}\underline{F}\left(  x\right)  /\mathrm{d}x,} & {a^{\prime\prime
}<x<b_{1}}%
\end{array}
\right.  ,\ F_{0}=\alpha  ,\ F_{1}=\underline
{F}\left(  {b_{1}}\right)  .
\]%
\[
F\left(  x\right)  =\left\{
\begin{array}
[c]{cc}%
{\alpha  ,} & {b_{0}\leqslant x\leqslant
a^{\prime\prime}}\\
{\underline{F}\left(  x\right)  ,} & {a^{\prime\prime}<x<b_{1}}%
\end{array}
\right.  .
\]
and the proof is similar to the one of above cases. Optimal shape of $F$ for any interval $[b_i,b_{i+1}]$ can be obtained by replacing $b_0$ and $b_1$ by respectively $b_i$ and $b_{i+1}$ in the above proofs, as they are general (as pictured on Fig.~\ref{fig:gencase}). All is left to prove is that  the concatenated $F$ obtained by the piece-wise extremizing solutions is increasing (i.e., that $F_i$ for $[b_{i-1},b_i]$ is lower or equal than  $F_i$ for $[b_{i},b_{i+1}]$).

\textbf{Step (2) of the proof}  Now we show that the joint extremizing distribution function is increasing.
Without loss of generality we consider only two intervals $\left[
{b_{0},b_{1}}\right]  $ and $\left[  {b_{1},b_{2}}\right]  $. The maximal
value of the function $F\left(  x\right)  $ in the interval $\left[
{b_{0},b_{1}}\right]  $ is $\max\left(  {\overline{F}\left(  {b_{0}}\right)
,\underline{F}\left(  {b_{1}}\right)  }\right)  $ for all the cases. The
minimal value of the function $F\left(  x\right)  $ in the interval $\left[
{b_{1},b_{2}}\right]  $ is $\min\left(  {\overline{F}\left(  {b_{1}}\right)
,\underline{F}\left(  {b_{2}}\right)  }\right)  $ for all the cases. 

If $\underline{F}\left(  {b_{2}}\right)  \geq\overline{F}\left(  {b_{0}%
}\right)  $, then
\[
\min\left(  {\overline{F}\left(  {b_{1}}\right)  ,\underline{F}\left(  {b_{2}%
}\right)  }\right)  \geq\max\left(  {\overline{F}\left(  {b_{0}}\right)
,\underline{F}\left(  {b_{1}}\right)  }\right)  .
\]
This means that the function is increasing. 

If $\underline{F}\left(  {b_{2}}\right)  <\overline{F}\left(  {b_{0}}\right)
$, then $\underline{F}\left(  {b_{1}}\right)  <\overline{F}\left(  {b_{0}%
}\right)  $ and we can take $F\left(  x\right)  =\underline{F}\left(  {b_{1}%
}\right)  $ for the left interval. On the other hand, $\underline{F}\left(
{b_{2}}\right)  <\overline{F}\left(  {b_{1}}\right)  $ and we can take
$F\left(  x\right)  =\overline{F}\left(  {b_{1}}\right)  $ for the left
interval. It follows from the condition $\underline{F}\left(  {b_{1}}\right)
<\overline{F}\left(  {b_{1}}\right)  $ that the function $F\left(  x\right)  $
is increasing in two neighbour intervals.

Figure~\ref{fig:generalcase-illu} gives an example of a general extremizing distribution.
\end{proof}

\begin{proof}
[\textbf{Proof using random sets}]For convenience, we
will consider that $h$ begins with a local minimum and ends with a
local maximum $a_{n}$. Formulas when $h$ begins (resp. ends) with a local
maximum (resp. minimum) are similar. Lower/upper
expectations can be computed as follows:%
\begin{align*}
\underline{\mathbb{E}}(h) &  =\int\limits_{0}^{\underline{F}(b_{n}%
)}\min_{b_{i}\in
A_{\gamma}}(h(a_{\ast\gamma}),h(b_{i}),h(a_{\gamma
}^{\ast}))d\gamma+\int\limits_{\underline{F}(b_{n})}^{1}%
h(a_{\ast\gamma})d\gamma,\\
\overline{\mathbb{E}}(h) &  =\int\limits_{0}^{\underline{F}(a_{1}%
)}h(a_{\gamma}^{\ast})d\gamma+\int\limits_{\underline
{F}(a_{1})}^{\overline{F}(a_{n})}\max_{a_{i}\in A_{\gamma}%
}(h(a_{\ast\gamma}),h(a_{i}),h(a_{\gamma}^{\ast}))d\gamma.
\end{align*}
We concentrate on the formula giving the lower expectation  (details
for upper one are similar). The most interesting part is the first
integral. We consider a particular level $\gamma$. Let $B=\{b_{i},\ldots,b_{j}\}\quad(i\leq j)$ be the set of
local minima included in the set $A_{\gamma}$ ($B$ can be
empty). $b_{i-1}$ and $b_{j+1}$ are the closest local minima outside
$A_{\gamma}$. We then consider the minimal $\Delta\gamma:= \gamma + \delta\gamma$ such that $\min_{b_{i}\in
A_{\gamma}}(h(a_{\ast\gamma}),h(b_{i}),h(a_{\gamma
}^{\ast})) \neq \min_{b_{i}\in
A_{\Delta\gamma}}(h(a_{\ast,\Delta\gamma}),h(b_{i}),h(a_{\Delta\gamma}^{\ast}))$ with $\min_{x \in A_{\Delta\gamma}}h(x) \neq h(a_{\ast,\Delta\gamma})$ if $\min_{x \in A_{\gamma}}h(x) = h(a_{\ast,\gamma})$ and $\min_{x \in A_{\Delta\gamma}}h(x) \neq h(a^{\ast}_{\Delta\gamma})$ if $\min_{x \in A_{\gamma}}h(x) = h(a^{\ast}_{\gamma})$. As in LP proof, four different cases can occur:

\textbf{Case A:} we have $$\min_{b_{i}\in
A_{\gamma}}(h(a_{\ast\gamma}),h(b_{i}),h(a_{\gamma
}^{\ast}))=h(b_k)$$ and $$\min_{b_{i}\in
A_{\Delta\gamma}}(h(a_{\ast,\Delta\gamma}),h(b_{i}),h(a_{\Delta\gamma}^{\ast}))=h(b_{k'}),$$ with $k\neq k'$ and where $h(b_k)$ and $h(b_{k'})$ are respectively the lowest local minima of $h(x)$ for $x \in A_{\gamma}$ and $x \in A_{\Delta\gamma}$. That is, probability mass is concentrated on $b_k$ from $\gamma$ to $\Delta\gamma$, and concentrates on $b_{k'}$ for values $\gamma' \geq \Delta\gamma$. This correspond to Case 1. of Fig.~\ref{fig:gencase} and of the previous proof. In Fig.~\ref{fig:generalcase-illu}, it corresponds to the extremizing distribution between $b_2$ and $b_3$.

\textbf{Case B:} we have $$\min_{b_{i}\in
A_{\gamma}}(h(a_{\ast\gamma}),h(b_{i}),h(a_{\gamma
}^{\ast}))=h(a_{\ast\gamma})$$ and $$\min_{b_{i}\in
A_{\Delta\gamma}}(h(a_{\ast,\Delta\gamma}),h(b_{i}),h(a_{\Delta\gamma}^{\ast}))=h(a_{\Delta\gamma}^{\ast}).$$
This can happen when any local minimum inside $A_{\gamma}$,$A_{\Delta\gamma}$ is higher than local minima just outside it. In this case, it can happen that minimal values stand at the bounds of intervals $A_{\gamma'}$ for any $\gamma \leq \gamma' \leq \Delta\gamma$. This corresponds to Case 2.1. of Fig.~\ref{fig:gencase} and of the previous proof. In Fig.~\ref{fig:generalcase-illu}, it corresponds to the extremizing distribution between $b_4$ and $b_5$.

\textbf{Case C:} we have $$\min_{b_{i}\in
A_{\gamma}}(h(a_{\ast\gamma}),h(b_{i}),h(a_{\gamma
}^{\ast}))=h(b_k)$$ and $$\min_{b_{i}\in
A_{\Delta\gamma}}(h(a_{\ast,\Delta\gamma}),h(b_{i}),h(a_{\Delta\gamma}^{\ast}))=h(a_{\Delta\gamma}^{\ast}).$$
With $h(b_k)$ the lowest local minima for $b_k \in A_{\gamma}$. The minimum shift from the left bound of $A_{\gamma}$ (coinciding with $\udf$) to $b_k$. This corresponds to Case 2.2. of Fig.~\ref{fig:gencase} and of the previous proof. In Fig.~\ref{fig:generalcase-illu}, it corresponds to the extremizing distribution between $b_1$ and $b_2$.

\textbf{Case D:} we have $$\min_{b_{i}\in
A_{\gamma}}(h(a_{\ast\gamma}),h(b_{i}),h(a_{\gamma
}^{\ast}))=ha_{\ast\gamma})$$ and $$\min_{b_{i}\in
A_{\Delta\gamma}}(h(a_{\ast,\Delta\gamma}),h(b_{i}),h(a_{\Delta\gamma}^{\ast}))=h(b_{k'}).$$
With $h(b_{k'})$ the lowest local minima for $b_{k'} \in A_{\Delta\gamma}$. Situation is similar to the previous case, and corresponds to Case 2.3. of Fig.~\ref{fig:gencase} and of the previous proof. In Fig.~\ref{fig:generalcase-illu}, it corresponds to the extremizing distribution between $b_3$ and $b_4$.

When $\min_{b_{i}\in
A_{\gamma}}(h(a_{\ast\gamma}),h(b_{i}),h(a_{\gamma
}^{\ast}))=\min_{b_{i}\in
A_{\Delta\gamma}}(h(a_{\ast\gamma}),h(b_{i}),h(a_{\gamma
}^{\ast})) = h(b_k)$ with $b_k \in A_{\gamma} \cap A_{\Delta\gamma}$, probability mass stay concentrated on $b_k$, and this corresponds to a discontinuity mentioned in Proposition~\ref{prop:manymaxmin}. By letting $\gamma$ evolve from $0$ to $1$, we get the extremizing cumulative distribution of Proposition~\ref{prop:manymaxmin}.
\end{proof}

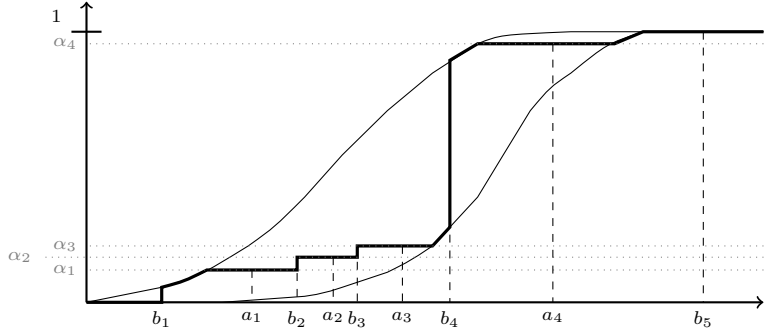
\begin{figure}
\begin{tikzpicture}[scale=2]

\draw[ thick, ->] (0,0) -- (45mm,0mm) ; \draw[thick]
(-1mm,18mm) node[above left] {$\scriptstyle 1$} -- (1mm,18mm) ;
\draw[thick, ->] (0,0) -- (0mm,20mm) ;

 \draw (0,0) -- (5mm,1mm) .. controls (7.5mm,1.75mm) ..
(10.75mm,3.75mm) .. controls (12.5mm,5mm) .. (14.5mm,7mm) --
(17mm,9.8mm) -- (20mm,12.75mm) -- (23mm,15.25mm) -- (26mm,17.2mm) ..
controls (27.5mm,17.8mm) .. (32.2mm,18mm) -- (37mm,18mm) --
(45mm,18mm) ;

\draw (0,0) -- (9mm,0mm) -- (14.5mm,0.4mm) .. controls
(15.75mm,0.6mm) .. (17mm,1mm) -- (20mm,2mm) .. controls
(21.5mm,2.75mm) .. (23mm,3.75mm) -- (26mm,7mm) -- (27.7mm,9.8mm) --
(29.5mm,12.75mm) .. controls (30.8mm,14.3mm) .. (32.2mm,15.25mm) ..
controls (34.6mm,17.2mm) .. (37mm,18mm) -- (45mm,18mm) ;

\draw[very thick] (0,0) -- (5mm,0mm) -- (5mm,1mm) .. controls
(6.5mm,1.4mm) .. (8mm,2.15mm) -- (14mm,2.15mm) -- (14mm,3mm) --
(18mm,3mm) -- (18mm,3.75mm) -- (23mm,3.75mm) -- (24.154mm,5mm) --
(24.154mm,16.1mm) -- (26mm,17.2mm) -- (35.1mm,17.2mm) -- (37mm,18mm)
-- (45mm,18mm) ;

\draw[dotted,thin,gray] (45mm,2.15mm) -- (0mm,2.15mm) node[left]
{$\scriptstyle \alpha_1$} ; \draw[dotted,thin,gray] (45mm,3mm) --
(-3mm,3mm) node[left] {$\scriptstyle \alpha_2$} ;
\draw[dotted,thin,gray] (45mm,3.75mm) -- (0mm,3.75mm) node[left]
{$\scriptstyle \alpha_3$} ; \draw[dotted,thin,gray] (45mm,17.2mm) --
(0mm,17.2mm) node[left] {$\scriptstyle \alpha_4$} ;

\draw[dashed] (5mm,1mm) -- (5mm,0mm) node[below] {$\scriptstyle
b_1$} ;

\draw[dashed] (14mm,3mm) -- (14mm,0mm) node[below] {$\scriptstyle
b_2$} ;

\draw[dashed] (18mm,3.75mm) -- (18mm,0mm) node[below] {$\scriptstyle
b_3$} ;

\draw[dashed] (24.154mm,16.1mm) -- (24.154mm,0mm) node[below]
{$\scriptstyle b_4$} ;

\draw[dashed] (41mm,18mm) -- (41mm,0mm) node[below] {$\scriptstyle
b_5$} ;

\draw[dashed] (11mm,2.15mm) -- (11mm,0mm) node[below] {$\scriptstyle
a_1$} ;

\draw[dashed] (16.4mm,3mm) -- (16.4mm,0mm) node[below]
{$\scriptstyle a_2$} ;

\draw[dashed] (21mm,3.75mm) -- (21mm,0mm) node[below] {$\scriptstyle
a_3$} ;

\draw[dashed] (31mm,17.2mm) -- (31mm,0mm) node[below] {$\scriptstyle
a_4$} ;

%\draw[dotted,ultra thick] (4mm,0.5mm)  rectangle (14.5mm,4.5mm) ;

%\draw (4mm,4.5mm) node[above left] {A} ;

%\draw[dotted,ultra thick] (13mm,1.5mm)  rectangle (19mm,6mm) ;

%\draw (13mm,6mm) node[above left] {C} ;

%\draw[dotted,ultra thick] (17.5mm,2.5mm)  rectangle (26mm,7mm) ;

%\draw (26mm,7mm) node[above right] {B} ;

%\draw[dotted,ultra thick] (22mm,15mm)  rectangle (38mm,19mm) ;

%\draw (22mm,19mm) node[above left] {D} ;

\end{tikzpicture}
\caption{Example of Optimal F with general
$h$}
\label{fig:generalcase-illu}
\end{figure}

Looking at the extremizing distribution $F$ pictured in Figure~\ref{fig:generalcase-illu}, we can see that computing the lower expectation consists in concentrating  probability masses over local minima, while giving the less possible amount of probability mass to higher values of $h(x)$, as in the case of a function having one maximum. Thus, our results confirm what could have intuitively be guessed at first sight. They also give analytical and computational tools to compute lower and upper expectations. They are illustrated in the next example. 

\begin{example}
\label{exmp:general}
We consider the same p-box $\pbox$ as in the previous examples (see Example~\ref{exmp:mono-uncon}). However, we assume that the loss function is of the type $h(x)=(0.6x)\cos(x)$. It could, for instance, model the return of a game based on the movement of a pendulum. It could also model the loss incurred by a unit failure whose functioning alternate between low and full capacity (failure during low capacity periods costing less). As a loss after failure has to be positive, one can consider $h(x)+\mu$, with $\mu$ a positive constant\footnote{This does not change further calculations, as $\underline{\mathbb{E}}(h + \mu)=\underline{\mathbb{E}}(h)+\mu$.}. $h(x)$ is oscillating between local maxima and minima. These extrema are solutions of  $\cos
(x)=x\sin(x)$:
\[
a_{1}=0.860,\ b_{1}=3.426,\ a_{2}=6.\allowbreak437,\ b_{2}=9.529,\ a_{3}%
=12.645,
\]%
\[
b_{3}=15.771,\ a_{4}=18.\allowbreak902,\ b_{4}=22.036,\ a_{5}=25.172,\ b_{5}%
=28.31.
\]
We will compute  the extremizing distribution for each intervals $[b_i,b_{i+1})$ for $i=1,\ldots,5$, with $b_0=0$. Let us analyze the first interval $[0,b_{1})$. The value $\alpha\in(0,1)$ in
this interval can be found as a root of the equation
\begin{align*}
&  \left(  \max\left(  -2\ln(1-\alpha),0\right)  \right)  \cdot\cos
(\max\left(  -2\ln(1-\alpha),0\right)  )\\
&  =\left(  \min\left(  -5\ln(1-\alpha),3.426\right)  \right)  \cdot\cos
(\min\left(  -5\ln(1-\alpha),3.426\right)  ).
\end{align*}
However, many different values of $\alpha \in (0,1)$ are solutions to the above equations. Relying on the proof of Proposition~\ref{prop:manymaxmin} and on the various subcases exposed therein (see Fig.~\ref{fig:gencase}), we should, for a given interval $[b_i,b_{i+1})$,  take only root(s) which provides the
interval $[a^{\prime},a^{\prime\prime}]$ such that $a_{i}%
\in\lbrack a^{\prime},a^{\prime\prime}]$. For $[0,b_{1})$, this corresponds to $\alpha=0.215$, for which  values $a^{\prime}%
,a^{\prime\prime}$ are 
\[
a^{\prime}=\max\left(  -2\ln(1-\alpha),b_{i-1}\right)  =\max\left(
-2\ln(1-0.215),0\right)  =\allowbreak0.483,
\]%
\[
a^{\prime\prime}=\min\left(  -5\ln(1-\alpha),b_{i}\right)  =\min\left(
-5\ln(1-0.215),3.426\right)  =\allowbreak1.209.
\]
It can be seen from the above that $a_{1}=0.860\in\lbrack0.483,\allowbreak
1.209]$. We can now determine the extremizing distribution function in $[0,b_{1})$, which is as
follows:%
\[
F\left(  x\right)  =\left\{
\begin{array}
[c]{ll}%
1-\exp(-0.5\cdot x), & x<\allowbreak0.483\\
0.215, & \allowbreak0.483\leq x\leq1.209\\
1-\exp(-0.2\cdot x), & 1.209<x<3.426
\end{array}
\right.  .
\]
This corresponds to the case 2.1. of Figure~\ref{fig:gencase}. the "jump" (i.e., probability mass) at point $b_{1}$ is of the size
\[
\min\left(  1-\exp(-0.5\cdot3.426),0.808\right)  -\max\left(  1-\exp
(-0.2\cdot3.426),0.215\right)  =\allowbreak0.312.
\]
Since $\overline{F}(3.426)-\underline{F}%
(3.426)=0.33>\allowbreak0.312$, this means that the extremizing distribution in $[b_1,b_{2})$ starts with a constant value $F(b_1)=\underline{F}%
(3.426)+0.312=0.808$ and with an horizontal line. Moreover, we can check that $0.808$ is the right starting point since it is a root of the equation 
\begin{align*}
&  \max\left(  -2\ln(1-\alpha),3.426\right)  \cdot\cos(\max\left(
-2\ln(1-\alpha),3.426\right) \\
&  =\min\left(  -5\ln(1-\alpha),9.529\right)  \cdot\cos(\min\left(
-5\ln(1-\alpha),9.529\right)  .
\end{align*}
And we have $a^{\prime}=\allowbreak3.426$ and $a^{\prime\prime}=\allowbreak8.263$ for $\alpha=0.808$. By
taking into account the analysis of the first interval, we can write
\[
F\left(  x\right)  =\left\{
\begin{array}
[c]{ll}%
0.808, & \allowbreak3.426\leq x\leq8.263\\
1-\exp(-0.2\cdot x), & 8.263<x<9.529
\end{array}
\right.  .
\]
This correspond to case 2.3. of Figure~\ref{fig:gencase}. the jump at $b_{2}$ has value $\allowbreak
9.77\times10^{-2}$, and we have again $\overline{F}(9.529)-\underline{F}%
(9.529)=\allowbreak0.14>\allowbreak9.77\times10^{-2}$. Analysis for other intervals are similar (they all belong to case 2.3.). For the third interval
$[b_{2},b_{3})$, $\alpha=0.948$, $a^{\prime}=\allowbreak
9.529$, $a^{\prime\prime}=14.831$ and we have
\[
F\left(  x\right)  =\left\{
\begin{array}
[c]{ll}%
0.949, & \allowbreak\allowbreak\allowbreak9.\,\allowbreak529\leq
x\leq\allowbreak14.\,\allowbreak831\\
1-\exp(-0.2\cdot x), & 14.\,\allowbreak831<x<15.771
\end{array}
\right.  .
\]
The jump at $b_{3}$ is of value $\allowbreak2.867\times10^{-2}$, and for $[b_{3},b_{4})$, we have $\alpha=0.986$, $a^{\prime}=\allowbreak15.771\allowbreak$, $a^{\prime\prime}=\allowbreak
21.255$ and 
\[
F\left(  x\right)  =\left\{
\begin{array}
[c]{ll}%
0.986, & \allowbreak\allowbreak\allowbreak15.\allowbreak771\leq x\leq
21.255\allowbreak\\
1-\exp(-0.2\cdot x), & 21.255<x<22.036
\end{array}
\right.  .
\]
The jump at $b_{4}$ is of value $\allowbreak8.189\times10^{-3}$, and for $[b_{4},b_{5})$, we have $\alpha=0.996$, $a^{\prime}=22.036\allowbreak$, $a^{\prime\prime}=27.62$ and 
\[
F\left(  x\right)  =\left\{
\begin{array}
[c]{ll}%
0.996, & \allowbreak\allowbreak\allowbreak\allowbreak22.036\leq x\leq
27.62\allowbreak\\
1-\exp(-0.2\cdot x), & 27.62<x<28.31
\end{array}
\right.  .
\]
The jump at point $b_{5}$ is of the size $\allowbreak\allowbreak
3.076\times10^{-3}$. 

Note that jump sizes decrease as index $i$ increase. This is not true in general, and is here due to the particular shape of $h(x)$. By computing the extremizing distribution for every interval $[b_{i-1},b_{i})$, we can reach the lower expectation. That is, if we note $\underline{\mathbb{E}}_{i}(h)$ the lower expectation of $h$ computed with the extremizing distribution obtained for $i$
intervals $[b_{j-1},b_{j}), j=1,\ldots,i$, and if $h$ have a finite number of local maxima
and minima, say $r$, then $\underline{\mathbb{E}}(h)=\underline{\mathbb{E}%
}_{r}(h)$. However, in this example, $r=\infty$ and $\underline{\mathbb{E}}%
(h)=\lim_{r\rightarrow\infty}\underline{\mathbb{E}}_{r}(h)$. Therefore, only an approximate solution can be found\footnote{We assume here that the expectation
$\underline{\mathbb{E}}(h)$ exists.}. We can therefore let $r$ increase until $\left\vert \underline{\mathbb{E}}%
_{r}(h)-\underline{\mathbb{E}}_{r-1}(h)\right\vert \leq\varepsilon$,
with $\varepsilon>0$ a prescribed precision. For instance, we have%
\begin{align*}
\underline{\mathbb{E}}_{1}(h)  & =\int_{0}^{\allowbreak0.483}0.6x\cos
(x)\cdot0.5e^{-0.5x}\mathrm{d}x\\
& +\int_{1.209}^{\allowbreak3.426}0.6x\cos(x)\cdot0.2e^{-0.2x}\mathrm{d}x\\
& +\allowbreak0.6\cdot\allowbreak3.426\cos(\allowbreak3.426)\cdot0.312\\
& =\allowbreak-0.82.
\end{align*}
Pursuing the computations, we have
$$\underline{\mathbb{E}}_{2}(h)= -1.558, \quad  \underline{\mathbb{E}}_{3}(h)= -1.9, \quad \underline{\mathbb{E}}_{4}(h)= -2.033, \quad \underline{\mathbb{E}}_{5}(h)= -2.093.$$
If we take $\varepsilon=0.1$, then $\left\vert \underline{\mathbb{E}}_{5}(h)-\underline{\mathbb{E}}%
_{4}(h)\right\vert =\allowbreak0.06<0.1$, and we consider $\underline{\mathbb{E}}_{5}(h)=-2.093$ as a sufficient approximation of the true (but unknown) lower approximation. Upper expectation of $h$ can be obtained by considering the function $-h(x)$ and by computing $\underline{\mathbb{E}}(-h)$. Hence $\overline{\mathbb{E}}(h)=-\underline{\mathbb{E}}(-h)=1.94$ (approximation with $\varepsilon=0.1$).
\end{example}

This example is useful in two respects: first, it illustrates why it is useful to have results concerning the piece-wise extremizing distribution; second, it shows that even when analytical calculations are possible, it is not always possible to compute an exact value, hence the interest of the generic methods proposed in Section~\ref{sec:probstat}. This is particularly true when $h$ has an infinity of local extrema and when $\udf,\ldf$ have infinite support. It also addresses the question of the choice of levels $\alpha$ when many solutions are possible.

Coming back to numerical approximations using linear programming, our results indicates that some regions should be sampled in priority. For example, when computing lower expectations, one should primarily consider values $b_i$ (local minima) and sample in neighbourhoods of these values, as it is where probability masses are concentrated. The converse (sampling around local maxima) holds when computing upper expectations.

If we now consider random set, we can formulate the problem of computing lower expectations as follows: let $m$ be the number of local minima, and let $\gamma_{j_{\ast}},\gamma_{j^{\ast}}$ be the two values bounding the probability mass concentrated on local minima $b_{j}$, for $j=1,\ldots,m$ (for example, for the local minima $b_{2}$ in Figure~\ref{fig:generalcase-illu}, we would have $\gamma_{2_{\ast}}=\alpha_1,\gamma_{2^{\ast}}=\alpha_2$), then 
\begin{equation}
 \underline{\mathbb{E}}(h)=\sum\limits_{j=1}^m(\int\limits_{\gamma_{(j-1)^{\ast}}%
}^{\gamma_{j_{\ast}}}\min (h(a_{\ast\gamma}),h(a_{\gamma
}^{\ast})) d\gamma+(\gamma_{(j)^{\ast}%
}-\gamma_{j_{\ast}})h(b_{j})). \label{eq:manymaxlowgen}%
\end{equation}
This comes down to sum all the probability masses concentrated on local minima, and to calculate integrals when the extremizing distribution coincide either with $\udf$ or $\ldf$. Note that, as in Example~\ref{exmp:general}, $m$ could be equal to $\infty$. This formulation clearly shows that, when using numerical methods with the random set approach, there is no need to discretize in finer intervals the intervals $[\gamma_{j_{\ast}},\gamma_{(j)^{\ast}}]$, as it won't improve the precision of the result.

The case of conditional expectation with general function will not be treated here, as it would require long development that wouldn't bring many new ideas.

\section{Conclusions}

We have considered the problem of computing lower and upper
expectations on p-boxes and particular functions under two different
approaches: by using linear programming and by using the fact that
p-boxes are special cases of random sets. Although the two
approaches try to solve equivalent problems, their differences
suggest different ways to approximate the solutions of those
problems. As we have seen, knowing the behaviour of the function over which lower and upper expectations are to be estimated can greatly increase the computational efficiency (and even permit analytical computation).

However, more important than their differences is the complementarity of both
approaches. Indeed, one approach can shed light on some problems
obscured by the other approach (e.g., the level $\alpha$ of
proposition~\ref{pr:NonMonF1}). Another advantage of combining both
approaches is the ease with which some problems are solved and the
elegant formulation resulting from this combination (e.g., the
conditional case). Let us nevertheless note that the constraint
programming approach can be applied to imprecise probabilities in
general, while the random set approach is indeed limited to random
sets.

In this paper, we have concentrated on the case where uncertainty bears on one variable. 
The case where multiple variables are tainted with uncertainty described by p-boxes will be studied in a forthcoming paper. Concerning future work related to this topic, three lines of research seem interesting to us:
\begin{itemize}
\item study of other simple representations : it is desirable to achieve similar studies for other simple uncertainty representations involving sets of probabilities. This includes probability intervals~\cite{CamposAll94}, possibility distributions~\cite{DuboisPrade88}, clouds~\cite{Neumaier04}.
\item Discretization schemes : when exact solutions cannot be computed, what is the best choice of points $x_1,\ldots,x_N$ or of levels $\gamma_1,\ldots,\gamma_M$, respectively to approximate the solution by using LP or RS (already mentioned by other authors~\cite{Tonon08}). We have mentioned how our results can possibly help in this task, but proposing generic algorithms and empirically testing them largely remains to be done.
\item Convex mixture of functions : in some applications, one can choose a strategy that is a convex mixture between a finite set of options having utility $h_1,\ldots,h_N$. For such cases, one often has to find the weights $\lambda_1,\ldots,\lambda_N$ such that $\sum_{i=1,N} \lambda_i h_i$ have the maximal lower expectation. It would be interesting to study whether similar results as the ones exposed in this paper also exists for this problem when using simple uncertainty representations (e.g., p-boxes).
\end{itemize}

We would like to end this paper with two final remarks:
\begin{itemize}
\item it is clear from our results  that extreme distributions over which the upper and lower expectations will be reached will be, in general, discontinuous. Since any discontinuous functions can be approximated as close as one wants by continuous ones, we do not see it as a big flaw. However, in some cases, it could be desirable to add constraints about which cumulative distributions inside $\pbox$ are admissible. This kind of questions is adressed, for example, by Kozine and Krymsky~\cite{KozineKrymsky07}.
\item We mention at the beginning of the paper that our study is restricted to the case where either cumulative distributions were assumed to be $\sigma$-additive or where $h$ was continuous. Again, this is not a big limitation when dealing with practical applications, and this avoids many mathematical subtleties arising with the consideration of finitely additive probabilities~\cite{miranda2006c}.
\end{itemize}

\bibliographystyle{plain}
\bibliography{LevSeb_longver}

\end{document}